%% file: main.tex
\documentclass[onefignum,onetabnum]{siamart190516}

\usepackage{amsfonts}
\usepackage{bbm}
\usepackage{tikz}
\usetikzlibrary{shapes}
\usetikzlibrary{shapes.geometric}
\usetikzlibrary{positioning}
\usetikzlibrary{calc, arrows.meta}
\usepackage{pgfplots}
\pgfplotsset{compat=1.18}
\usepackage[inline]{enumitem} 
\usepackage{mathtools}

\usepackage{makecell}
\usepackage[most]{tcolorbox}

\usepackage{xparse}

\def\exampletext{Example} 

\NewDocumentEnvironment{example}{ O{} }
{
\colorlet{colexam}{red!55!black} 
\newtcolorbox[use counter=example]{testexamplebox}{%
    empty,
    title={\exampletext: #1},
    attach boxed title to top left,
    boxed title style={empty,size=minimal,toprule=0pt,top=4pt,left=3mm,overlay={}},
    coltitle=colexam,fonttitle=\bfseries,
    before=\par\medskip\noindent,parbox=true,boxsep=0pt,left=3mm,right=0mm,top=2pt,breakable,pad at break=0mm,
       before upper=\csname @totalleftmargin\endcsname0pt, 
    overlay unbroken={\draw[colexam,line width=.5pt] ([xshift=-0pt]title.north west) -- ([xshift=-0pt]frame.south west); },
    overlay first={\draw[colexam,line width=.5pt] ([xshift=-0pt]title.north west) -- ([xshift=-0pt]frame.south west); },
    overlay middle={\draw[colexam,line width=.5pt] ([xshift=-0pt]frame.north west) -- ([xshift=-0pt]frame.south west); },
    overlay last={\draw[colexam,line width=.5pt] ([xshift=-0pt]frame.north west) -- ([xshift=-0pt]frame.south west); },%
    }
\begin{testexamplebox}}
{\end{testexamplebox}\endlist}

\newcommand{\V}[1]{\boldsymbol{#1}} 
\newcommand{\abs}[1]{\left|#1\right|} 
\usepackage{fancybox}
\usepackage{multirow}

\newcommand{\bs}{{\boldsymbol{s}}}
\newcommand{\bx}{{\boldsymbol{x}}}
\newcommand{\bt}{{\boldsymbol{t}}}
\newcommand{\valpha}{{\boldsymbol{\alpha}}}
\newcommand{\psigma}{{\bs_{\partial R}}}
\newcommand{\rvalpha}{{\tilde{\boldsymbol{\alpha}}}}
\newcommand{\ralpha}{{\tilde{\alpha}}}

\tikzset{every picture/.style={line width=1pt}} 

\input{ex_shared}
\usepackage { graphicx,epstopdf } 
\usepackage{subcaption}

\ifpdf
\hypersetup{
  pdftitle={},
  pdfauthor={}
}
\fi

\begin{document}


\maketitle

\begin{abstract}
The branching algorithm is a fundamental technique for designing fast exponential-time algorithms to solve combinatorial optimization problems exactly.  
It divides the entire solution space into independent search branches using predetermined branching rules, and ignores the search on suboptimal branches to reduce the time complexity. 
The complexity of a branching algorithm is primarily determined by the branching rules it employs, which are often designed by human experts.
In this paper, we show how to automate this process with a focus on the maximum independent set problem. The main contribution is an algorithm that efficiently generate optimal branching rules for a given sub-graph with tens of vertices.
Its efficiency enables us to generate the branching rules on-the-fly, which is provably optimal and significantly reduces the number of branches compared to existing methods that rely on expert-designed branching rules.
Numerical experiment on 3-regular graphs shows an average complexity of $O(1.0441^n)$ can be achieved, better than any previous methods.

\end{abstract}




\section{Introduction}\label{sec:intro}

The branching algorithm~\cite{morrison2016branch}, also called the branch-and-bound algorithm, is a standard scheme to produce exact solutions to combinatorial optimization problems~\cite{peres2021combinatorial}.
Its main idea is to divide the solution space into smaller regions, and conquer each subproblem independently and recursively. 
Since its birth in 1960s~\cite{land1960branch}, the branching algorithm has been applied to solving many NP-hard problems~\cite{Moore2011}, such as the integer programming~\cite{le2017abstract, khalil2016learning, gamrath2018measuring}, the maximum satisfiability (MAX-SAT) problem~\cite{li2005exploiting, abrame2014ahmaxsat, argelich2018clause}, and the traveling salesman problem (TSP)~\cite{eppstein2007traveling, carpaneto1980some}.
It is safe to say that at least half of the published fast exponential time algorithms are based on this scheme~\cite{Fomin2013}.

One of the most successful applications of the branching algorithm is exactly solving the maximum independent set (MIS) problem~\cite{tarjan1977finding}.
With a given graph, the MIS problem asks for the largest set of vertices such that no two vertices are adjacent.
The MIS problem is classified as NP-complete~\cite{Moore2011}, indicating that no algorithm can solve it in polynomial time unless $\text{P} = \text{NP}$. Furthermore, if the exponential time hypothesis holds true, it is widely believed that solving it in sub-exponential time is impossible~\cite{Lokshtanov2013,Impagliazzo2001}.
The branching algorithm gives the state-of-the-art complexity of $O(1.1996^n)$~\cite{Xiao2017}, where $n$ is the number of vertices in the graph.
By upper bounding the vertex degrees, an even lower complexity can be achieved. For example, the MIS problem on the 3-regular graphs can be solved in time $O(1.0836^n)$~\cite{Xiao2009,xiao2013}.
The branching algorithm studies a given graph by assuming some vertices to be in or out of the independent set, and then recursively solves each smaller subproblem under that assumption.
Branching rules play a central role in making assumptions. Usually, the algorithm prepares a fixed table of rules such as the mirror rule~\cite{Fomin2006}, the satellite rule~\cite{kneis2009fine}, and more specialized rules~\cite{jian1986algorithm, Fomin2006, bourgeois2012fast, robson1986algorithms}.
The algorithm inspects a sub-graph and implements the best-fit rule to divide the solution space into smaller ones.
The complexity of a branching algorithm is typically expressed as $\mathcal{O}(\gamma^n)$, where $\gamma > 1$ represents the \textit{branching factor}.
A good set of \textit{branching rules} gives a smaller $\gamma$, which leads to a faster algorithm.

However, these methods exhibit two disadvantages. 
First, they rely on a fixed set of rules, which requires human experts to design~\cite{xiao2013, Xiao2017}. 
Second, these rules are generic and not optimal for every specific sub-graph under analysis. 
To address these issues, we propose to branch in an on-the-fly manner, i.e. automatically generate branching rules for each sub-graph under consideration, rather than relying on a finite set of predefined rules. 
With a branching rule tailored for the given sub-graph, a provably optimal $\gamma$ and a faster algorithm can be achieved.


In this paper, an algorithm for automatically generating optimal branching rules for a given sub-graph is developed, which exhibits provably optimal complexity, i.e. generating the smallest $\gamma$ among all possible branching rules.
The algorithm first resolves the local constraints of a sub-graph with existing numerical tools. By establishing a connection between the problem of finding the optimal branching rules and a weighted minimum set covering problem (WMSC), the optimal branching rule can be efficiently resolved through integer programming or its linear programming relaxation.
With the MIS problem, we demonstrate that the larger the sub-graph the algorithm analyzes, the more effective the generated branching rule becomes.
By limiting the sub-graph to the 2-distance neighborhood of a vertex, the numerical result shows a clear advantage compared to existing methods in terms of reducing the number of branches. 
We also show how the developed tool can be used in theorem proving to reduce human effort.
In the past, the efforts to automate the branching algorithm are mainly focused on dynamically determining branching variables~\cite{achterberg2005branching,fischetti2011backdoor}.
To the best of our knowledge, our method is the first one that can dynamically generate new branching rules based on the variables under consideration.

The rest of the paper is structured as follows.
In \Cref{sec:preliminaries}, we introduce the preliminaries of the branching algorithm and the MIS problem.
In \Cref{sec:optimal-branching}, we show the details of the proposed method and its application to the MIS problem.
In \Cref{sec:rule-discovery}, we show several examples to demonstrate the effectiveness of the proposed method in discovering the branching rule.
In \Cref{sec:numerical-results}, numerical results are provided to show the efficiency of the proposed method in practice.
Finally, we conclude the paper in \Cref{sec:conclusion}.

\section{Preliminaries}\label{sec:preliminaries}

In this section, we introduce the basic concepts and notations to be used in this article, including the maximum independent set (MIS) problem, and the branching algorithm.

\subsection{Maximum Independent Set (MIS) problem}\label{sec:mis}
In graph theory, a graph $G$ is represented as a tuple of vertices $V(G)$ and edges $E(G)$.
The neighbors of a vertex $v$ in graph $G$ is $N(v) \subseteq V(G)$ and the neighbors of a set of vertices $S \subseteq V(G)$ is $N(S) = \bigcup_{v \in S} N(v) \setminus S$.
To simplify the discussion, we also introduce the closed neighbor of $S$ as $N[S] = N(S) \cup S$.
Similarly, the $k$-th order neighbors of $v$ is denoted as $N_k(v) = N(N_{k-1}[v])$, where $N_{1}[v]=N[v]$ and $N_k[v] = N_k(v) \cup N_{k-1}[v]$.
An independent set of a graph $G$ is a subset of vertices $I \subseteq V(G)$ such that no two vertices of $I$ are direct neighbors, i.e. $v, w \in I \Rightarrow v \notin N(w)$.
A maximum independent set (MIS) is an independent set of maximum cardinality, and its size is denoted as $\alpha(G)$.
Finding an independent set with size $\alpha(G)$ is known as the maximum independent set problem, which is an NP-hard problem~\cite{Moore2011}.
In the following discussion, we represent a subset of $V$ as a bit string $\bs_{V} \in \{0, 1\}^{|V|}$, where $s_i = 1$ if the $i$-th vertex is in the set, and $s_i = 0$ otherwise.


\subsection{The Branching Algorithm}\label{sec:branching-algorithm}


The branching algorithm is a general algorithmic framework for solving combinatorial optimization problems, including the MIS problem.
It iteratively constructs the solution by breaking the problem down into smaller sub-problems and then solves the sub-problems independently.
To break down the problem of finding an MIS of a given graph $G$, the branching algorithm checks the constraints on a sub-graph $R \subseteq G$ that is represented as a triple of the vertex set $V(R) \subseteq V(G)$, edge set $E(R) \subseteq E(G)$ and the boundary vertices connected to the rest part of the graph $\partial R \subseteq V(R)$.
These constraints determine the specific subset of the $2^{|V(R)|}$ local configurations that require further exploration, referred to as the \emph{relevant} configurations in the subsequent discussion.
A branching strategy decides how to divide the search space to get the relevant configurations explored efficiently.
As we denoted a configuration as a bit string, a branching strategy for bitstring searching can be concisely represented as a boolean formula in disjunctive normal form (DNF):
\begin{definition}[Branching strategy for bitstring searching]\label{def:branching-strategy}
    A branching strategy for bitstring searching denoted as $\delta$, is a function that maps a subgraph $R$ to a boolean formula in disjunctive normal form (DNF) $\mathcal{D} = c_1 \lor c_2 \lor \ldots \lor c_{|\mathcal{D}|}$, where each clause $c_k = l_1\land l_2\land \ldots\land l_{|V(c_k)|}$ is associated with a branch, in which a positive or negative literal $l_i = v$ or $l_i = \neg v$ represents the vertex $v \in V(R)$ is or isn't in the set.
    Here, $V(c_k)$ is the set of vertices involved in the $k$-th clause $c_k$.
\end{definition}
Here, we use a DNF with ${|\mathcal{D}|}$ clauses to divide the solution space into ${|\mathcal{D}|}$ sub-problems, each with a reduced problem size due to fixing the values of variables involved in the clause.
In the independent set problem, the size reduction comes from two aspects: the variables involved in the clause, and the neighbors of the vertices in the set.
Let us denote the vertices associated with a positive literal as $T(c_k) \subseteq V(c_k)$.
Then the problem size reduction should take into account the removal of vertices $V(c_k) \cup N(T(c_k))$ and the branching rule corresponding to $\delta$ is $\alpha(G)=\mathrm{max}(\{\alpha(G \backslash (V(c_k) \cup N(T(c_k)))) + \abs{T(c_k)} \mid c_k \in \mathcal{D} \})$.
To quantify the problem size reduction, we introduce the measure $\rho$ of problem complexity on a graph $G$,
under which the branching complexity is defined as follows:

\begin{definition}[Branching complexity of MIS]\label{def:branching-complexity}
    Given a graph $G$, a measure of the computational complexity $\rho$ and a branching rule $\mathcal{D} = c_1 \lor c_2 \lor \ldots \lor c_{|\mathcal{D}|}$, the branching complexity $\gamma \geq 1$ for $\mathcal{D}$ is determined by the following relation:
    \begin{equation}
        \gamma^{\rho(G)} = \sum_{i = 1}^{|\mathcal{D}|} \gamma^{\rho(G) - \Delta\rho(c_i)}\;.
    \end{equation}
    where the difference of the measure by the $i$-th clause is defined as $\Delta\rho(c_i) = \rho(G) - \rho(G \setminus (N(T(c_i)) \cup V(c_i)))$, and $\V{b}(\delta, R) = (\Delta\rho(c_1), \Delta\rho(c_2), \ldots, \Delta\rho(c_{|\mathcal{D}|}))$ is the branching vector.
    The optimal branching rule is defined as the one that minimizes branching complexity.
\end{definition}

The branching algorithm induced by the branching strategy $\delta$ is summarized in \Cref{alg:branching}.
Its overall complexity is upper bound by the maximum branching complexity of the branching rules used in the algorithm.
The choice of the complexity measure is highly related to the performance of the algorithm.
In this work, two complexity measures are used, one is the number of vertices in the graph, i.e. $\rho(G) = \abs{V}$, and the other is defined as~$\rho(G) = \sum_{v \in V} \max\{0, d(v) - 2\}$, where~$d(v) = \abs{N(v)}$ is the degree of vertex $v$. In the second measure, a graph with maximum degree $2$ has complexity $0$.
It is because the MIS problem on such a graph can be solved in polynomial time.
Although the vertices selecting strategy \texttt{select\_subgraph} is also important for the performance of the branching algorithm, it will not be the focus of this article.

\begin{algorithm}[!ht]
    \caption{The Branching algorithm for MIS: \texttt{mis\_branch}}
    \label{alg:branching} 
    \SetKwFunction{KwFn}{mis\_branch}
    \SetKwProg{Fn}{function}{}{end}
    \KwIn{The input graph $G$ and the branching strategy $\delta$}
    \KwOut{The maximum independent set size $\alpha(G)$}
    \Fn{\KwFn{$G$, $\delta$}}{
        $R \leftarrow \texttt{select\_subgraph}(G)$\tcp*{select by some strategy}

        $\mathcal{D} \leftarrow \delta(R)$\;

        $\alpha \leftarrow -\infty$\;

        \ForEach{{\rm clause} $c$ in $\mathcal{D}$}{
            $\alpha_c \leftarrow \texttt{mis\_branch}(G \setminus (V(c) \cup N(T(c))), \delta$) + $|T(c)|$\;

            \If{$\alpha_c > \alpha$}{
                $\alpha \leftarrow \alpha_c$\;
            }
        }
        \Return $\alpha$\;
    }
\end{algorithm}

\section{Numerical Methods for Optimal Branching}\label{sec:optimal-branching}
In this section, we will introduce our method for finding the optimal branching rule.
Since we have formalized the branching rules as a logic expression in DNF in \Cref{def:branching-strategy}, naively, we can search through all \emph{valid} logic expressions in DNF, and then evaluate the branching complexity for each rule to find the optimal one.
In the following, we will give a clear definition of a valid logic branching rule by introducing the concept of the reduced $\alpha$-tensor, the minimum subset of local configurations that need to be explored in the MIS problem.
Finally, we will show how to convert the optimal branching problem to a weighted minimum set covering problem, which can be solved efficiently using linear programming.

\subsection{Boundary-grouped MISs}\label{sec:boundary-grouped-mis}

We first introduce the concept of the \emph{$\alpha$-tensor} of a subgraph $R$ in $G$. 
The $\alpha$-tensor can be interpreted as a generalization of the scalar $\alpha(G)$, which denotes the MIS size of a graph $G$.
\begin{definition}[$\alpha$-tensor~\cite{Liu2024}]\label{def:alpha}
    Let $R$ be a subgraph in $G$ and $\partial R$ its inner boundary vertices.
    The \textit{$\alpha$-tensor} of $R$, $\valpha(R)$, is a tensor of rank $|\partial R|$, whose element $\valpha(R)_{\psigma}$ is  the size of the largest independent set of $R$, while fixing the boundary configuration $\psigma \in \{0,1\}^{|\partial R|}$.
    If the boundary-vertex configuration $\psigma$ itself violates the independent set constraint, the corresponding $\alpha$-tensor element is set to $\valpha(R)_{\psigma} = -\infty$.
\end{definition}

This $\alpha$-tensor is not ideal for designing branching rules, since it unnecessarily contains many elements that are irrelevant to finding an MIS.
In \Cref{sec:alpha-tensor}, we show that these irrelevant elements can be removed without leading to a non-optimal solution.
The reduced $\alpha$-tensor is defined as follows:

\begin{definition}[reduced $\alpha$-tensor]\label{def:reduced-alpha}
    Let $\valpha(R)$ be an $\alpha$-tensor for a subgraph $R$ of $G$, its corresponding reduced $\alpha$-tensor~$\rvalpha(R)$ is defined by setting all entries in $\valpha(R)$ that correspond to configurations that are irrelevant to finding an MIS to $-\infty$.
\end{definition}

Note that each finite element in the reduced $\alpha$-tensor is associated with a boundary configuration $\psigma \in \{0, 1\}^{|\partial R|}$,
and each such configuration corresponds to one or more configurations on $V(R)$ with the same local MIS size.
We group all relevant local configurations by the boundary configuration as the \emph{boundary-grouped MISs} $\mathcal{S}_R$:
\begin{definition}[boundary-grouped MISs]
A boundary-grouped MISs on $R$ is a set, each element $S_{\psigma}$ being a set of configurations on $V(R)$ with the same boundary configuration $\psigma$ and the local MIS size $\ralpha(R)_{\psigma}$, i.e.
\begin{equation}\label{eq:configuration-set}
    \begin{split}
    \mathcal{S}_R = \{S_{\V{s}_{\partial R}} \mid \V{s}_{\partial R} \in \{0, 1\}^{|\partial R|}, \ralpha(R)_{\V{s}_{\partial R}} \neq -\infty\},\\
    S_{\V{s}_{\partial R}} = \{\V{s} \mid {\rm Con}(\{\V{s}, \V{s}_{\partial R}\}), \ralpha(R)_{\V{s}_{\partial R}} = \abs{T(\V{s})}\},
    \end{split}
\end{equation}
where ${\rm Con}(\{\V{s}, \V{s}_{\partial R}\})$ denotes that the configuration $\V{s}$ is consistent with the boundary configuration $\V{s}_{\partial R}$.
\end{definition}
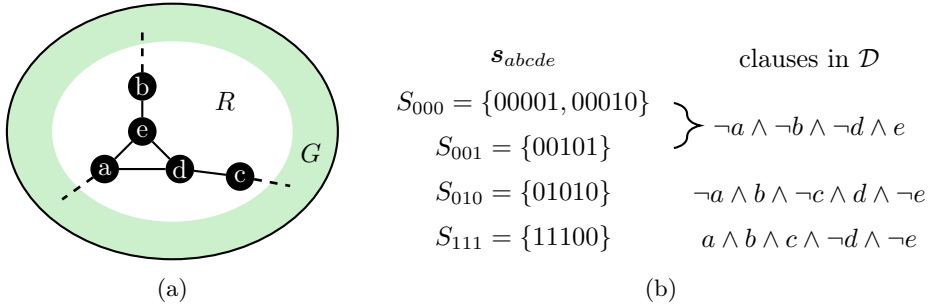
\begin{figure}[ht!]
    \centering
    \begin{subfigure}[b]{0.39\textwidth}
        \centering
        \begin{tikzpicture}[decoration=brace]
            \fill[green!70!black!50,opacity=0.4] (0,0) ellipse (2.2 and 1.7);
            \draw[thick] (0,0) ellipse (2.2 and 1.7); 
        
            \fill[white!30] (-0.0,0) ellipse (1.6 and 1.2);
        
            \node[circle, fill=black, draw=black, thick, minimum size=10, inner sep=0] at (-0.9, -0.5) (a) {\color{white}{a}};
            \node[circle, fill=black, draw=black, thick, minimum size=10, inner sep=0] at (-0.4, 0.6) (b) {\color{white}{b}};
            \node[circle, fill=black, draw=black, thick, minimum size=10, inner sep=0] at (0.9, -0.6) (c) {\color{white}{c}};
            \node[circle, fill=black, draw=black, thick, minimum size=10, inner sep=0] at (0.1, -0.5) (d) {\color{white}{d}};
            \node[circle, fill=black, draw=black, thick, minimum size=10, inner sep=0] at (-0.4, 0.0) (e) {\color{white}{e}};
            \draw[thick] (a) -- (e);
            \draw[thick] (a) -- (d);
            \draw[thick] (b) -- (e);
            \draw[thick] (c) -- (d);
            \draw[thick] (d) -- (e);

            \node[] at (-1.6, -1) (f) {};
            \draw[dashed] (a) -- (f);

            \node[] at (-0.4, 1.5) (g) {};
            \draw[dashed] (b) -- (g);

            \node[] at (1.7, -0.75) (h) {};
            \draw[dashed] (c) -- (h);
        
            \textcolor{black}{\node at (0.7, 0.4) {$R$};}
            \textcolor{black}{\node at (1.84,-0.28) {$G$};}
        \end{tikzpicture}
        \caption{}
    \end{subfigure}
    \begin{subfigure}[b]{0.59\textwidth}
        \centering
        \begin{tikzpicture}[scale = 1]
            \node at (2.5, 1.8) {$\bs_{abcde}$};
            \node at (2.5, 1.2) (node_1) {$S_{000} = \{00001, 00010\}$};
            \node at (2.5, 0.6) (node_2) {$S_{001} = \{00101\}$};
            \node at (2.5, 0.0) (node_3) {$S_{010} = \{01010\}$};
            \node at (2.5, -0.6) (node_4) {$S_{111} = \{11100\}$};


            \node at (6.3, 1.8) {clauses in $\mathcal{D}$};
            \draw[decorate, decoration={brace, amplitude=10pt, mirror}, thick] (4.5, 0.6) -- (4.5, 1.2);
            \node[] at (6.3, 0.9) (node_6) {\color{black}{$\neg a \land \neg b \land \neg d \land e$}};
            \node at (6.3, 0.0) (node_3) {\color{black}{$\neg a \land b \land \neg c \land d \land \neg e$}};
            \node at (6.3, -0.6) (node_4) {\color{black}{$ a \land b \land c \land \neg d \land \neg e$}};

          \end{tikzpicture}
        \caption{}
    \end{subfigure}
    \caption{
        Applying a MIS branching strategy on a sub-graph $R$ of the parent graph $G$.
        (a) A sub-graph $R$, with boundary vertices $\partial R = \{a, b, c\}$, where the dashed lines indicate their connections to vertices in the environment $G \backslash R$.
        (b) The left side shows the boundary-grouped MISs $\mathcal{S}_R = \{S_{000}, S_{001}, S_{010}, S_{111}\}$ of the sub-graph $R$. Each row represents a set $S_{\V{s}_{\partial R}}$ of relevant configurations with the boundary configuration $\V{s}_{\partial R}$.
        The right side shows the clauses in the resulting optimal branching $\mathcal{D} = \delta(R)$.
        }
    \label{fig:branching}
\end{figure}

\begin{example}[Reduced $\alpha$-tensor and boundary-grouped MISs]
    The reduced $\alpha$-tensor of the sub-graph shown in \Cref{fig:branching} is shown in the following table. Each row is associated with a boundary-vertex configuration $\psigma \in \{0,1\}^3$.
    The $\alpha$-tensor is shown in the second column, which corresponds to the local MIS size of the sub-graph $R$ under the boundary configuration $\psigma$.
    The third column shows the reduced $\alpha$-tensor, where the irrelevant entries are set to $-\infty$.
    The relevant configurations in the fourth column form the boundary-grouped MISs $\mathcal{S}_R = \{\{00001, 00010\}, \{00101\}, \{01010\}, \{11100\}\}$.

    \centering
    \begin{tabular}{|c|c|c|c|c|}
    \hline
        $\bs_{abc}$ & $\valpha(R)_\psigma$ & $\rvalpha(R)_\psigma$ & $\bs_{abcde}$ & Configuration index\\
        \hline
        000 & $1$ & $1$ & $00001, 00010$ & $1$ \\
        001 & $2$ & $2$ & $00101$ & $2$ \\
        010 & $2$ & $2$ & $01010$ & $3$ \\
        011 & $2$ & $-\infty$ & - & - \\
        100 & $1$ & $-\infty$ & - & - \\
        101 & $2$ & $-\infty$ & - & - \\
        110 & $2$ & $-\infty$ & - & - \\
        111 & $3$ & $3$ & $11100$ & $4$ \\
        \hline
    \end{tabular}
    \captionof{table}{The $\alpha$-tensor (\Cref{def:alpha}) (second column) and the reduced $\alpha$-tensor (\Cref{def:reduced-alpha}) $\rvalpha(R)$ (third column) for the sub-graph $R$ in \Cref{fig:branching}. Each row corresponds to a local MIS size associated with the boundary-vertex configuration $\bs_{abc}$ (first column). The fourth column lists the corresponding relevant configurations producing the tensor elements.}\label{tbl:example-R}    
\end{example}


\begin{theorem}\label{thm:boundary-grouped}
    The boundary-grouped MISs $\mathcal{S}_R$ of a sub-graph $R$ can be obtained in time $\min(O(2^{{\rm tw}(\overline{R})}), O(1.4423^{|V(R)|-\abs{\partial R}}2^{|\partial R|}))$, where $\overline{R}$ is completion of $R$ by adding a vertex that connecting to all the boundary vertices of $R$ and ${\rm tw}(\overline{R})$ is the tree-width of $\overline{R}$.
\end{theorem}
\begin{proof}
    The two time complexities are from two methods to obtain the boundary-grouped MISs, the generic tensor network method~\cite{Liu2023} and a branching algorithm respectively.
    The computational complexity of the generic tensor network method is determined by the topology of the tensor network.
    For the independent set problem with open boundary, the complexity of contracting a tensor network is related to the tree-width of the graph $\overline{R}$ as $O(2^{{\rm tw}(\overline{R})})$~\cite{Markov2008}. 
    The boundary-grouped MISs can also be obtained with the \texttt{mis1} algorithm~\cite{Fomin2013}, which has time complexity $O(1.4423^n)$ for a graph with $n$ vertices. For each of the $2^{\abs{\partial R}}$ boundary configurations, a MIS problem of size $\leq |V(R)|-\abs{\partial R}$ is solved, rendering an overall time complexity $O(1.4423^{|V(R)|-\abs{\partial R}}2^{|\partial R|}))$.
\end{proof}

Empirically, the tensor-network-based approach is favored by sparse graphs and geometric graphs, since they have small tree widths~\cite{Fomin2006b}. The branching algorithm is favored by the graph with high connectivity, since the branching complexity on vertex $v$ is related to its degree $d(v)$ as $\gamma \sim d(v)^{1/d(v)}$, i.e. the larger $d(v)$, the smaller $\gamma$ for $d(v) \geq 3$.


\subsection{Optimal Branching via Set Covering}\label{sec:optimal-branching-set-covering}

In this section, we show how to obtain the optimal branching rule given the boundary-grouped MISs $\mathcal{S}_R$.
Let $R$ be a sub-graph to be considered. By fixing the boundary vertices configuration $\bs_{\partial R}$, the optimal solutions in $R$ and the rest part of the graph $G\setminus R$ are decoupled, where $G\setminus R$ denotes the induced subgraph of $V(G)\setminus V(R)$. The global maximum independent set can be obtained by taking an arbitrary maximum independent set from each of the two components.
One naive strategy of branching is by boundary configurations in $\mathcal{S}_R$.
For each boundary configuration $\psigma$, we create a branch by fixing the local configurations to an arbitrary configuration in $S_{\V{s}_{\partial R}}$.
Interestingly, although this strategy fully utilizes the sparsity of reduced $\alpha$-tensor, it turns out to be non-optimal with high probability.
To achieve optimality, we must allow some variables in $V(R)$ to be undecided in the current branch.
Recall that a branching rule $\mathcal{D}$ is a boolean formula in DNF, we just need to search the space of a boolean formula in DNF and choose a \emph{valid} one with the lowest branching complexity.

\begin{definition}[Valid branching rule]
    A branching rule $\mathcal{D}$ is valid on $\mathcal{S}_R$ if and only if for any set $S_{\V{s}_{\partial R}} \in \mathcal{S}_R$, there exists a configuration $\bs_{V(R)} \in S_{\V{s}_{\partial R}}$ that satisfies $\mathcal{D}$, denoted as $S_{\V{s}_{\partial R}} \vdash \mathcal{D}$.
\end{definition}

It corresponds to the requirement that for each $S_{\V{s}_{\partial R}} \in \mathcal{S}_R$, at least one of its elements must be explored in one of the branches.
Finding the optimal branching rule via brute-force search is impractical, as the number of boolean formulas in DNF grows super-exponentially as $O(2^{3^{|V(R)|}})$.
In the following, we show that enumerating all valid boolean formulas in DNF can be done efficiently by formulating the problem as a WMSC problem.

\begin{definition}[Weighted minimum set covering problem]\label{def:wmsc}
    A weighted minimum set covering (WMSC) problem is a triple $(U, \mathcal{S}, w)$, where $U$ is a set of elements, $\mathcal{S}$ is a collection of subsets of $U$, and $w$ is a weight function that assigns a non-negative weight to each element in $\mathcal{S}$.
    The goal is to find a subset $I \subseteq \mathcal{S}$ that covers all elements in $U$ and minimizes the total weight of the selected subsets.
\end{definition}

To construct a valid branching rule, we first prepare a set of candidate clauses $\mathcal{C}$.
Here, we emphasize that not all $3^{|V(R)|}$ possible clauses on $V(R)$ can be used in the optimal branching rule, e.g. if clause $c_a$ covers the same set of items in $\mathcal{S}_R$ with less number of literals than clause $c_b$, then $c_a$ can not be used for constructing the optimal branching rule.
One possible way to filter out the infeasible clauses is through \Cref{alg:candidate_clauses}.
The algorithm starts with constructing clauses satisfied by exactly one configuration in one of $S_{\V{s}_{\partial R}} \in \mathcal{S}_R$.
This is achieved by the function $\texttt{single\_cover}$ that takes a configuration~$\V{s}$ as input and returns a clause $c$ given by
$$\bigwedge_{v \in V(R)} (v \text{ if }s_v =1 \text{ else } \neg v ).$$
Then, the algorithm iteratively takes the intersection of the clauses with the previously constructed clauses to form new clauses.
The function $\texttt{intersection}$ takes two clauses as input and returns a new clause that contains the common literals of the two clauses.
For example, $\texttt{intersection}(\neg a \land b \land c, \neg a \land  b \land \neg c \land d) = \neg a \land b$, such that any configuration satisfying the two input clauses will also satisfy the new clause. Although both $\neg a$ and $b$ are also satisfied by the same set of configurations, they are not selected since only the clause with the longest length is kept. This exclusion significantly reduces the number of candidate clauses without sacrificing the optimality of the branching rule.

\begin{algorithm}[ht]
    \caption{Generating candidate clauses: \texttt{candidate\_clauses}}
    \label{alg:candidate_clauses} 
    \SetKwFunction{KwFn}{candidate\_clauses}
    \SetKwProg{Fn}{function}{}{end}
    \KwIn{The boundary-grouped MISs $\mathcal{S}_R$}
    \KwOut{The candidate clauses~$\mathcal{C}$}
    \Fn{\KwFn{$\mathcal{S}_R$}}{

        $\mathcal{C} \leftarrow \emptyset$\;
        
        $\mathcal{T} \leftarrow \emptyset$\;

        \ForEach{{\rm{set of configurations}} $S$ in $\mathcal{S}_R$}{
            \ForEach{{\rm configuration} $\V{s}$ in $S$}{

                $c \leftarrow \texttt{single\_cover}(\V{s})$\;

                $\mathcal{C} \leftarrow \mathcal{C} \cup \{c\}$\;

                $\mathcal{T} \leftarrow \mathcal{T} \cup \{c\}$\;
            }
        }


        \While{$\mathcal{T} \neq \emptyset$}{
            
            $c \leftarrow \texttt{pop}(\mathcal{T})$\;

            \ForEach{{\rm{set of configurations}} $S$ in $\mathcal{S}_R$}{
                \ForEach{{\rm configuration} $\V{s}$ in $S$}{


                    $c^\prime = \texttt{intersection}(c, \texttt{single\_cover}(\V{s}))$\;

                    \If{$c^\prime \neq \emptyset$ \rm{and} $c^\prime \notin \mathcal{C}$}{
                        $\mathcal{C} \leftarrow \mathcal{C} \cup \{c^\prime\}$\;

                        $\mathcal{T} \leftarrow \mathcal{T} \cup \{c^\prime\}$\;
                    }
                }
            }
        }

        \Return $\mathcal{C}$\;
    }
\end{algorithm}

Then we represent a candidate solution $\mathcal{D}$ as a disjunction of a subset of clauses in $\mathcal{C}$,
and the subset can be denoted as a bit string $\bx \in \{0, 1\}^{|\mathcal{C}|}$, where we use $x_i = 1$ to denote the $i$-th clause is included in $\mathcal{D}$, and $x_i = 0$ otherwise.
To relate the optimal branching rule with the WMSC problem, we associate each candidate clause $c_i$ with an integer set $J_i = \{j \mid S_j \vdash c_i, j=1,\ldots, |\mathcal{S}_R|\}$ to indicate the elements in $\mathcal{S}_R$ that can be satisfied by $c_i$.
Then finding the branching rule with the minimum complexity can be formulated as the following optimization problem:
\begin{equation}\label{eq:gamma}
    \begin{split}
        \min_{\gamma, \bx} \gamma ~~ \text{ s.t. } & \sum_{i=1}^{|\mathcal{C}|} \gamma^{-\Delta\rho(c_i)} x_i = 1,\\
        & \bigcup\limits_{\substack{i = 1,\ldots, |\mathcal{C}|\\ x_i = 1}} J_i = \{1, 2, \ldots, |\mathcal{S}_R|\}
    \end{split}
\end{equation}
where $\Delta\rho(c_i)$ is the reduction of the problem size in the branch induced by the clause $c_i$.
The first constraint ensures that the branching complexity is the same as the $\gamma$ defined in \Cref{def:branching-complexity} and the second constraint ensures that the branching rule explores all the configurations in $\mathcal{S}_R$ and is valid.
To solve \Cref{eq:gamma}, we first consider a dual problem: given a fixed branching complexity $\gamma$, find a solution to $\V{x}$ that satisfies
\begin{equation}\label{eq:set-covering-constraint}
    \min_{\bx} \sum_{i=1}^{|\mathcal{C}|} \gamma^{-\Delta\rho(c_i)} x_i ~~ \text{ s.t. } \bigcup\limits_{\substack{i = 1,\ldots,{|\mathcal{D}|}\\ x_i = 1}} J_i = \{1, 2, \ldots, |\mathcal{S}_R|\}.
\end{equation}
It corresponds to the following WMSC problem (see \Cref{def:wmsc}):
\begin{equation}\label{eq:set-covering}
    (\{1, 2, \ldots, |\mathcal{S}_R|\}, \{J_1, J_2, \ldots, J_{|\mathcal{C}|}\}, i \mapsto \gamma^{-\Delta\rho(c_i)})\;,
\end{equation}
where the weight function is $w(i) = \gamma^{-\Delta\rho(c_i)}$.
With an oracle to solve this problem, we show the minimum branching complexity $\gamma$ can be resolved to very high precision in logarithmic time.

\begin{theorem}\label{thm:fast-convergence}
    Let $R$ be a subgraph of $G$, the branching strategy $\delta(R)$ with the smallest branching complexity $\gamma$ can be approximated to precision $\epsilon$ in time $O(\log(\epsilon^{-1}))$.
\end{theorem}
\begin{proof}
    We define an indicator function $\sigma(\gamma)$ that returns $1$ if the WMSC problem in \Cref{eq:set-covering} renders a valid solution, and $0$ otherwise:
    \begin{equation}
        \sigma(\gamma) = \begin{cases}
        1, & \text{ if } \exists \V{x}^*,~\text{s.t. }\parbox[t]{.6\textwidth}{$\sum_{i=1}^{|\mathcal{C}|} \gamma^{- \Delta\rho(c_i^*)} x_i^* \leq 1$, where\\ $\V{x}^*$ satisfying the constraints in  \Cref{eq:set-covering}\\ 
        }\\
        0, & \text{ otherwise}.
        \end{cases}
    \end{equation}
    $\sigma(\gamma)$ is a step function that monotonically increases with $\gamma$, the transition point of which can be found by bisecting the interval $[1, 2]$ in $O(\log(\epsilon^{-1}))$ time. This transition point is the minimum branching complexity $\gamma$.
\end{proof}

The \Cref{thm:fast-convergence} provides an logarithmic time method to obtain an approximation of the minimum branching complexity $\gamma$.
However, an exact solution is desired in practice to ensure the optimality of the branching rule.
In practice, we utilize a fixed point iteration method to directly obtain the exact solution of the WMSC problem, allowing us to converge on the minimum value of $\gamma$. This approach is preferred over the logarithmic time method outlined in \Cref{thm:fast-convergence}.
The corresponding algorithm is summarized in \Cref{alg:opt_branching}. 
The function \texttt{boundary\_grouped} returns the boundary-grouped MISs for $R$, which can be implemented with the generic tensor network method or the most basic branching algorithm as discussed in \Cref{thm:boundary-grouped}.
The function $\texttt{wmsc\_solver}$ solves the WMSC problem in \Cref{eq:gamma}. Although this problem is NP-complete, it can be solved efficiently in practice by reducing it to an integer programming problem as will be discussed in the following section.
In each iteration, the solver provides a valid solution $\V{x}$ for the WMSC problem, parameterized by $\gamma_{\rm old}$. Subsequently, we determine a new value of $\gamma \leq \gamma_{\rm old}$ that satisfies the equation $\sum_{i = 1}^{|\mathcal{D}|} {\gamma}^{-\Delta \rho(x_i)} x_i = 1$, using a root-finding subroutine. Since the left-hand side of this equation is monotonically decreasing within the interval $[1, 2]$, we can guarantee the existence and uniqueness of the root. We then update $\gamma_{\rm old}$ with this new value of $\gamma$ and proceed to the next iteration. The minimum value of $\gamma$ is precisely identified when both $\gamma$ and $\V{x}$ converge to the same value during the fixed point iteration. For a comprehensive proof of convergence, please refer to \Cref{sec:fixed_point}.

\begin{algorithm}[!ht]
    \caption{The optimal branching strategy: $\delta$}
    \label{alg:opt_branching} 
    \SetKwFunction{KwFn}{$\delta$}
    \SetKwProg{Fn}{function}{}{end}
    \KwIn{The input graph $G$ and a sub-graph $R \subseteq G$}
    \KwOut{The optimal branching $\mathcal{D}$ on $R$}

    \Fn{\KwFn{G, R}}{
        $\mathcal{S} \leftarrow \texttt{boundary\_grouped}(R)$\;

        $\mathcal{C} \leftarrow \texttt{candidate\_clauses}(\mathcal{S})$\;

        $\V{J} \leftarrow \{\{j \mid S_j \vdash c_i, j=1,\ldots, |\mathcal{S}_R|\} \mid c_i \in \mathcal{C}\}$\;

        \tcp{search for the optimal branching complexity $\gamma$}
        $\gamma \leftarrow 2$
        
        $\gamma_{old} \leftarrow \infty$\;

        \While{$\gamma < \gamma_{old}$}{

            \tcp{The triple $(\{1, \ldots, {|\mathcal{S}_R|}\}, \V{J}, i \mapsto \gamma^{-\Delta\rho(c_i)})$ defines a WMSC problem}
            $\V{x} \leftarrow \texttt{wmsc\_solver}(\{1, \ldots, {|\mathcal{S}_R|}\}, \V{J}, i \mapsto \gamma^{-\Delta\rho(c_i)})$\;

            $\gamma_{old} \leftarrow \gamma$\;

            $\gamma \leftarrow \texttt{find\_root}(\sum_{i=1}^{|\mathcal{C}|} \gamma^{-\Delta\rho(c_i)} x_i - 1 = 0)$\\
        }

        $\mathcal{D} \leftarrow \mathop{\bigvee}\limits_{\substack{i=1,\dots, {|\mathcal{C}|}\\ x_i=1}}c_i$\;

        \Return $\mathcal{D}$\;
    }

\end{algorithm}

\begin{example}[Optimal branching]
    Continuing the previous example, we consider how to obtain the optimal branching for the boundary-grouped MISs in \Cref{tbl:example-R}. 
    For simplicity, here we use the number of vertices as the measure, i.e. $\Delta\rho(c_i) = |V(c_i) \cup N(T(c_i))|$ and ignore the removal of vertices outside of $R$, which provides a lower bound for the reduction.
    Then we list the corresponding candidate clauses in \Cref{tbl:example-branching}.
    It can be easily verified that the solution of the WMSC problem in \Cref{eq:set-covering} is $c_4 \lor c_5 \lor c_7$.
    It is a valid cover of all configuration indices: $J_4 \cup J_5 \cup J_7 = \{1, 2, 3, 4\}$, and the corresponding branching complexity is the root of the following equation:
    \begin{equation}
        \gamma^{-5} + \gamma^{-5} + \gamma^{-4} = 1,
    \end{equation}
    which is $\gamma \approx 1.2672$.
    
    \centering
    \begin{tabular}{|c|c|c|c|}
    \hline
        $i$ & $J_i$ & $c_i$ & $\Delta\rho(c_i)$\\
        \hline
        $1$ & $\{1\}$ & $ \neg a \land  \neg b \land  \neg c \land  \neg d \land  e$ & $5$ \\
        $2$ & $\{1\}$ & $ \neg a \land  \neg b \land  \neg c \land  d \land  \neg e$ & $5$ \\
        $3$ & $\{2\}$ & $ \neg a \land  b \land  \neg c \land  d \land  \neg e$ & $5$ \\
        $4$ & $\{3\}$ & $ \neg a \land  \neg b \land  c \land  \neg d \land  e$ & $5$ \\
        $5$ & $\{4\}$ & $ a \land  b \land  c \land  \neg d \land  \neg e$ & $5$ \\
        $6$ & $\{1, 2\}$ & $ \neg a \land  \neg c$ & $2$ \\
        $7$ & $\{1, 2\}$ & $ \neg a \land  \neg c \land  d \land  \neg e$ & $4$ \\
        $8$ & $\{1, 3\}$ & $ \neg a \land  \neg b \land  \neg d \land  e$ & $4$ \\
        $9$ & $\{1, 3\}$ & $ \neg a \land  \neg b$ & $2$ \\
        $10$ & $\{2, 4\}$ & $ b \land  \neg e$ & $2$ \\
        $11$ & $\{3, 4\}$ & $ c \land  \neg d$ & $2$ \\
        $12$ & $\{1, 2, 3\}$ & $ \neg a$ & $1$ \\
        $13$ & $\{1, 2, 4\}$ & $ \neg e$ & $1$ \\
        $14$ & $\{1, 3, 4\}$ & $ \neg d$ & $1$ \\
        \hline
    \end{tabular}
    \captionof{table} {
        The candidate clauses for the boundary-grouped MISs in \Cref{tbl:example-R}.
        The first column is the index of the candidate clause, the second column is the set of configuration indices (see the last column of \Cref{tbl:example-R}) covered by the clause, the third column is the boolean expression of the clause, and the fourth column is the reduction of measure $\Delta\rho(c_i)$ by the corresponding clause.
    }\label{tbl:example-branching}

\end{example}

\subsection{Solving WMSC Problem via Integer Programming}\label{sec:wmsc-integer-programming}

In this section, we will show how to solve the WMSC problem in \Cref{def:wmsc}.
Let us consider the specific WMSC problem defined in \Cref{eq:set-covering}, we can convert it to an integer programming problem:
\begin{equation}\label{eq:integer_programming}
    \begin{split}
        \min &\sum_{i=1}^{|\mathcal{C}|} \gamma^{-\Delta\rho(c_i)} x_i,\\
        \text{s.t.} &\sum_{i=1, j\in J_i}^{|\mathcal{C}|} x_i \geq 1, \forall j = 1, 2, \ldots, |\mathcal{S}_R|,\\
        &x_i \in \{0, 1\}, \forall i = 1, 2, \ldots, {|\mathcal{C}|}.
    \end{split}
\end{equation}
The boolean variables $x_i$ indicate whether the set $J_i$ is chosen or not in the WMSC problem. The objective function is the sum of the weights of the selected sets, and the linear constraints ensure that each element in $\mathcal{S}_R$ is covered by at least one selected set.
Although the integer programming problem is NP-hard, it can be effectively addressed in practice using the Julia package \texttt{JuMP.jl}~\cite{Dunning2017,Lubin2023}, which utilizes \texttt{SCIP}~\cite{Achterberg2008, Achterberg2009} as its backend. This combination is recognized as one of the fastest non-commercial solvers for mixed integer programming (MIP).
Typically, it enables us to solve a WMSC problem exactly with about $10^3$ sets in less than~$10^{-2}$ seconds.
However, it is mainly based on the branch-and-bound algorithm, so that is exponential in time in the worst case.
In cases involving larger-scale problems, the integer programming formulation can be relaxed to a linear programming problem by substituting the binary variables $x_i$ with continuous variables constrained to the range $0 \leq x_i \leq 1$:
\begin{equation}
    \begin{split}
        \min &\sum_{i=1}^{|\mathcal{C}|} \gamma^{-\Delta\rho(c_i)} x_i,\\
        \text{s.t.} &\sum_{i=1, j\in J_i}^{|\mathcal{C}|} x_i \geq 1, \forall j = 1, 2, \ldots, |\mathcal{S}_R|,\\
        &0 \leq x_i \leq 1, \forall i = 1, 2, \ldots, {|\mathcal{C}|}. \\
    \end{split}
\end{equation}
The solution~$\V{x}$ of the linear programming problem can be interpreted as probability, and we randomly pick the clauses in $\mathcal{C}$ that satisfy all constraints.
It turns out to be a good approximation to the integer programming problem as we will show in the \Cref{sec:numerical-results}.
Unlike integer programming, a linear programming problem can be solved in time polynomial to the problem size~\cite{Cohen2021}.
In our implementation, we use the simplex algorithm~\cite{nabli2009overview} to solve the linear programming problem provided by the \texttt{HiGHS}~\cite{huangfu2018parallelizing} solver.

\section{Branching Rule Discovery}\label{sec:rule-discovery}

In this section, we will show how our method can be applied to automatic rule discovery with a few examples, including recovering the established rules and discovering new rules.
We show how the optimal branching rule discovery can help theorem proving by improving the branching rules for some bottleneck cases.

\subsection{Rediscovery of established rules}

Existing branching algorithms are constructed based on a variety of manually derived heuristic rules.
As one of them, the domination rule is an important rule for solving the MIS problem, which is also part of the \texttt{mis2} algorithm in Ref.~\cite{Fomin2013}.

\begin{lemma}[The Domination rule]
    Let~$G = (V, E)$ be a graph, $v$ and~$w$ be adjacent vertices of~$G$ such that~$N[v] \subseteq N[w]$.
    Then 
    \begin{equation}
        \alpha(G) = \alpha(G \setminus w)\;.
    \end{equation}
\end{lemma}

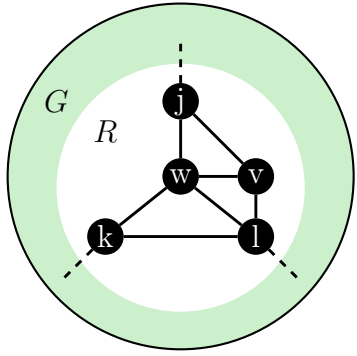
\begin{figure}[ht!]
    \centering
    \begin{tikzpicture}
    \fill[green!70!black!50,opacity=0.4] (0,0) ellipse (2.3 and 2.3);
            \draw[thick] (0,0) ellipse (2.3 and 2.3); 
        
            \fill[white!30] (-0.0,-0.15) ellipse (1.65 and 1.65);
        \node[circle, fill=black, draw=black, thick, minimum size=13,inner sep = 0,font=\fontsize{13pt}{12pt}\selectfont] at (0, 0) (w) {\color{white}{w}};
        \node[circle, fill=black, draw=black, thick, minimum size=13,inner sep = 0,font=\fontsize{13pt}{12pt}\selectfont] at (1, 0) (v) {\color{white}{v}};
        \node[circle, fill=black, draw=black, thick, minimum size=13,inner sep = 0,font=\fontsize{13pt}{12pt}\selectfont] at (0, 1) (j) {\color{white}{j}};
        \node[circle, fill=black, draw=black, thick, minimum size=13,inner sep = 0,font=\fontsize{13pt}{12pt}\selectfont] at (-1, -0.8) (k) {\color{white}{k}};
        \node[circle, fill=black, draw=black, thick, minimum size=13,inner sep = 0,font=\fontsize{13pt}{12pt}\selectfont] at (1, -0.8) (l) {\color{white}{l}};
        \draw[line width=1.1pt] (j) -- (w);
        \draw[line width=1.1pt] (j) -- (v);
        \draw[line width=1.1pt] (w) -- (v);
        \draw[line width=1.1pt] (w) -- (k);
        \draw[line width=1.1pt] (w) -- (l);
        \draw[line width=1.1pt] (k) -- (l);
        \draw[line width=1.1pt] (v) -- (l);

        \node[] at (0, 1.9) (m) {};
        \draw[dashed,line width=1.1pt] (j) -- (m);
        \node[] at (-1.7, -1.5) (n) {};
        \draw[dashed,line width=1.1pt] (k) -- (n);
        \node[] at (1.7, -1.5) (o) {};
        \draw[dashed,line width=1.1pt] (l) -- (o);
        \textcolor{black}{\node[font=\fontsize{12pt}{12pt}\selectfont] at (-1.0, 0.6) {$R$};}
            \textcolor{black}{\node[font=\fontsize{12pt}{12pt}\selectfont] at (-1.65,1) {$G$};}
    \end{tikzpicture}
    \caption{
        An example subgraph, where the $j$, $k$, $l$ are the boundary vertices.
        It satisfies the condition of domination rule, where~$N[v] \subseteq N[w]$.
    }
    \label{fig:domination}
\end{figure}

We will show how the optimal branching can automatically capture this rule by considering the subgraph shown in \Cref{fig:domination}.
Observing that $N[v] \subseteq N[w]$, according to the domination rule, the vertex~$w$ can be discarded.
In our method, we first generate the reduced $\alpha$-tensor and the boundary-grouped candidate local MISs, which is given in \Cref{tbl:domination}.
These configurations render a WMSC problem, and by solving it, we obtain the optimal DNF $\neg w$ since $S_{s_{jkl}} \vdash \neg w$ for all effective $s_{jkl}$, which is consistent with the domination rule.

\begin{table}[ht!]
    \centering
    \begin{tabular}{|c|c|c|}
    \hline
        $\bs_{jkl}$ & $\rvalpha(R)_\psigma$ & $\bs_{wvjkl}$\\
        \hline
        000 & 1 & 01000, 10000 \\
        010 & 2 & 01010 \\
        101 & 2 & 00101 \\
        111 & 3 & 00111 \\
        \hline
    \end{tabular}
    \caption{The finite-valued elements of reduced $\alpha$-tensor $\rvalpha(R)$ for $R$ in~\Cref{fig:domination}.}
\label{tbl:domination}    
\end{table}

\subsection{Discovery of optimal rules}
In this subsection, we will show how our method can automatically discover the optimal branching rule for a given subgraph, resulting in lower branching complexity compared to the state-of-the-art branching algorithm in Ref.~\cite{issac20131}.
We will use the example of PH2 (comprising a \textbf{p}entagon and a \textbf{h}exagon sharing \textbf{2} edges) in Ref.~\cite{issac20131} to demonstrate that our method achieves a better $\gamma$ than their heuristic rules.
As illustrated in Fig.~\ref{fig:ph2}, PH2 is a subgraph comprising 8 vertices $a$-$h$, where the vertices $a$, $c$, $d$, $f$, $g$ and $h$ serve as $\partial R$.
For such 3-regular graphs, we applied a commonly used measure given by~$\rho(G) = \sum_{v \in V} \max\{0, d(v) - 2\}$, since the vertices with $d(v) = 1, 2$ can be directly removed via graph rewriting~\cite{xiao2013}.
Here, we emphasize that the problem size reduction $\Delta\rho(c_i)$ should also take into account the vertices in $N_2[R]$, since the possible removal for vertices in $N_1[R]$ may decrease the degrees for vertices in $N_2(R)$. 
As an example, we talk about a neighborhood where no well-established rules exist, assume that the vertices within $N_1(R)$ are neither connected to each other nor share common neighbors, and that each vertex in $N_2[R]$ has a degree of 3. 
Without losing generality, we use the tree-like environment $N_3[R]$ shown in Fig.~\ref{fig:ph2}.
It has been shown that for this subgraph, the heuristic rule~\cite{issac20131} obtained manually can reach a branching complexity of $1.0718$, with a branching vector of $\{10,10\}$.

\begin{figure}[ht!]
    \centering
    \begin{tikzpicture}[decoration=brace,scale=0.7]

    \fill[green!70!black!50,opacity=0.4] (0,0) ellipse (4.5 and 4.8);
    \fill[white] (0,0) ellipse (3.4 and 3.6);
    \draw[thick] (0,0) ellipse (4.5 and 4.8); 
        \fill[yellow!90!black!80,opacity=0.5] (0,0) ellipse (3.4 and 3.6);
    
        \fill[white!30] (0,0) ellipse (2 and 2);
    
        \node[circle, fill=black, draw=black, thick, minimum size=10, inner sep=0] at (-0.0, 0.3) (a) {\color{white}{a}};
        \node[circle, fill=black, draw=black, thick, minimum size=10, inner sep=0] at (-0.8, -0.5) (b) {\color{white}{b}};
        \node[circle, fill=black, draw=black, thick, minimum size=10, inner sep=0] at (0.8, -0.5) (e) {\color{white}{e}};
        \node[circle, fill=black, draw=black, thick, minimum size=10, inner sep=0] at (-0.8, -1.4) (c) {\color{white}{c}};
        \node[circle, fill=black, draw=black, thick, minimum size=10, inner sep=0] at (0.8, -1.4) (d) {\color{white}{d}};
        \node[circle, fill=black, draw=black, thick, minimum size=10, inner sep=0] at (-1.4, 0.7) (f) {\color{white}{f}};
        \node[circle, fill=black, draw=black, thick, minimum size=10, inner sep=0] at (1.4, 0.7) (h) {\color{white}{h}};
        \node[circle, fill=black, draw=black, thick, minimum size=10, inner sep=0] at (-0.0, 1.7) (g) {\color{white}{g}};
        \draw[thick] (a) -- (b);
        \draw[thick] (a) -- (e);
        \draw[thick] (b) -- (c);
        \draw[thick] (e) -- (d);
        \draw[thick] (c) -- (d);
         \draw[thick] (b) -- (f);
          \draw[thick] (e) -- (h);
           \draw[thick] (g) -- (f);
            \draw[thick] (g) -- (h);
\newcommand{\treelike}{
  \begin{scope}
        \node[circle, fill=white, draw=black, thick, minimum size=7, inner sep=0] at (-0.0, -0.4)(k2) {};
        \node[circle, fill=white, draw=black, thick, minimum size=7, inner sep=0] at (-0.52, -0.7) (k3) {};
        \node[circle, fill=white, draw=black, thick, minimum size=7, inner sep=0] at (0.52, -0.7) (k4) {};
        \node[circle, fill=white, draw=black, thick, minimum size=7, inner sep=0] at (-0.22, -1.124) (k5) {};
        \node[circle, fill=white, draw=black, thick, minimum size=7, inner sep=0] at (-0.82, -1.124) (k6) {};
        \node[circle, fill=white, draw=black, thick, minimum size=7, inner sep=0] at (0.22, -1.124) (k7) {};
        \node[circle, fill=white, draw=black, thick, minimum size=7, inner sep=0] at (0.82, -1.124) (k8) {};
         \node[] at (-0.22, -2.1) (k9) {};
        \node[] at (-0.82, -2.1) (k10) {};
        \node[] at (0.22, -2.1) (k11) {};
        \node[] at (0.82, -2.1) (k12) {};
            \draw[thick] (k2) -- (k3);
    \draw[thick] (k2) -- (k4);
    \draw[thick] (k3) -- (k5);
    \draw[thick] (k3) -- (k6);
    \draw[thick] (k4) -- (k7);
    \draw[thick] (k4) -- (k8);
    \draw[dashed] (k5) -- (k9);
    \draw[dashed] (k6) -- (k10);
    \draw[dashed] (k7) -- (k11);
    \draw[dashed] (k8) -- (k12);
  \end{scope}
}
     \begin{scope}[xshift=-29,yshift=-45,rotate=-45]
    \treelike
    \draw[thick] (c) -- (k2);
  \end{scope}

  \begin{scope}[xshift=29,yshift=-45,rotate=45]
    \treelike
    \draw[thick] (d) -- (k2);
  \end{scope}  

   \begin{scope}[xshift=-48,yshift=20,rotate=-100]
    \treelike
    \draw[thick] (f) -- (k2);
  \end{scope} 
\begin{scope}[xshift=48,yshift=20,rotate=100]
    \treelike
    \draw[thick] (h) -- (k2);
  \end{scope}
  \begin{scope}[xshift=-10,yshift=54,rotate=210]
    \treelike
    \draw[thick] (g) -- (k2);
    
  \end{scope}
   \begin{scope}[xshift=20,yshift=52,rotate=-210]
    \treelike
    \draw[thick] (a) -- (k2);
    \coordinate (k12) at (1,3.6);
  \end{scope}

        \textcolor{black}{\node at (1.5, -0.5) {$R$};}
        \textcolor{black}{\node at (2.5, -1) {$N_3[R]$};}
        \textcolor{black}{\node at (3.6,-1.5) {$G$};}
    \end{tikzpicture}
    \caption{
        The example of PH2 (composed of a \textbf{p}entagon and a \textbf{h}exagon sharing \textbf{2} edges) with a tree-like neighborhood. The vertices connected by dashed lines indicate their connections to vertices in the further environment.
        } 
    \label{fig:ph2}
\end{figure}
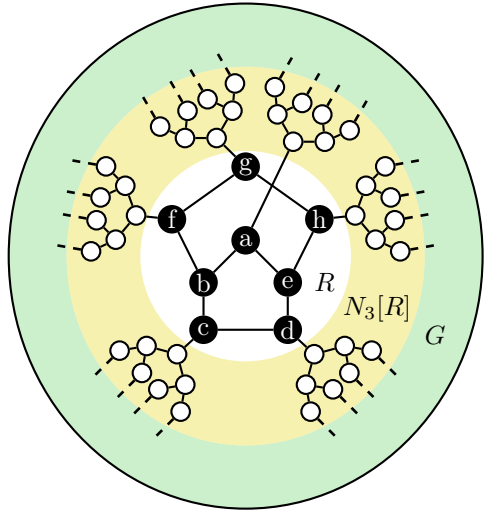

\begin{table}[ht!]
    \centering
    \begin{tabular}{|c|c|c|c|}
    \hline
        $\bs_{acdfgh}$ & $\rvalpha(R)_\psigma$ & $\bs_{abcdefgh}$ & Configuration index\\
        \hline
        010100 & 3 & 00101100 & 1 \\
        000010 & 3 & 01001010 & 2 \\
        001001 & 3 & 01010001 & 3 \\
        110101 & 4 & 10100101 & 4 \\
        101101 & 4 & 10010101 & 5 \\
        \hline
    \end{tabular}
    \caption{The finite-valued elements of reduced $\alpha$-tensor $\rvalpha(R)$ for $R=$ PH2 in Fig.~\ref{fig:ph2}.}
\label{tbl:ph2}    
\end{table}

\begin{table}[ht!]
    \centering
    \begin{tabular}{|c|c|c|c|}
    \hline
        $i$ & $J_i$ & $c_i$ & $\Delta\rho(c_i)$\\
        \hline
        $1$ & $\{1\}$ & $\lnot a \land \lnot b \land c \land \lnot d \land e \land  f \land \lnot g \land \lnot h$ & 18 \\
        $2$ & $\{2\}$ & $\lnot a \land b \land \lnot c \land \lnot d   \land e \land  \lnot f \land g \land \lnot h$ &  16 \\
        $3$ & $\{3\}$ & $\lnot a \land b \land \lnot c \land d \land \lnot e \land  \lnot f \land \lnot g \land h$ & 18 \\
        $4$ & $\{4\}$ & $a \land \lnot b \land c \land \lnot d \land \lnot e \land  f \land \lnot g \land h$ & 22 \\
        $5$ & $\{5\}$ & $a \land \lnot b \land \lnot c \land d \land \lnot e \land  f \land \lnot g \land h$ &  22 \\
        $6$ & $\{1,2\}$ &  $\lnot a \land \lnot d \land e \land \lnot h$ & 10 \\
        $7$ & $\{1,3\}$ & $\lnot a \land \lnot g$ & 8 \\
        $8$ & $\{1,4\}$ & $\lnot b \land c  \land \lnot d\land f \land \lnot g$ & 16 \\
        $9$ & $\{2,3\}$ & $\lnot a\land b  \land \lnot c\land \lnot f$ & 10 \\
        $10$ & $\{3,5\}$ & $\lnot c\land  d  \land \lnot e\land \lnot g \land h$ & 16 \\
       $11$ &  $\{4,5\}$ & $a\land \lnot b  \land \lnot e\land  f \land \lnot g \land h$ & 18 \\
        $12$ & $\{1,2,3\}$ & $\lnot a$ & 4 \\
        $13$ & $\{1,2,4\}$ & $\lnot d$ & 4 \\
        $14$ & $\{1,4,5\}$ & $\lnot b\land  f  \land \lnot g$ & 10 \\
        $15$ & $\{2,3,5\}$ & $\lnot c$ & 4 \\
        $16$ & $\{3,4,5\}$ & $\lnot e \land \lnot g\land h $ & 10 \\
        $17$ & $\{1,3,4,5\}$ & $\lnot g$ & 4 \\
        \hline
    \end{tabular}
    \caption{The candidate clauses for the boundary-grouped MISs in \Cref{tbl:ph2}.}
\label{subcovers:ph2}    
\end{table}

It turns out to be not optimal.
To show this, we start with the reduced $\alpha$-tensor and the boundary-grouped MISs for the PH2 subgraph as illustrated in \Cref{tbl:ph2}.
Then we generate all candidate clauses $\mathcal{C} = \{c_1, c_2, \ldots, c_{17}\}$ as shown in~\Cref{subcovers:ph2}.
The $J_i$ column indicates the configuration indices covered by $c_i$ and the $\Delta\rho(c_i)$ column indicates the reduction in the graph measure in the branch induced by $c_i$.
Finally, by solving the WMSC problem, we obtain the optimal branching rule $\mathcal{D}=c_2 \lor c_8 \lor c_{10}$ with branching vector $=\{16,16,16\}$.
The corresponding $\gamma=1.0711$, which is smaller than $1.0718$ of the heuristic rule as shown in \Cref{compare}.

\begin{table}[ht!]
    \centering
    \begin{tabular}{|c|c|c|c|}
    \hline
        $ $ & $\mathcal{D}$ & branching vector & $\gamma$\\
        \hline
        Manual rule in Ref.~\cite{issac20131} & $c_9 \lor c_{14}$ & $\{10,10\}$ & 1.0718 \\
        \hline
        Optimal rule & $c_2 \lor c_8 \lor c_{10}$ & $\{16,16,16\}$ & 1.0711\\
        \hline
    \end{tabular}
    \caption{Branching rules obtained by different methods.}
\label{compare}    
\end{table}

\subsection{Broadening the finite rule set}
Another example is from Ref.~\cite{xiao2013}, which gives the lowest complexity of $O(1.0836^n)$ for 3-regular graphs.
The rule with the highest complexity, or the \emph{bottleneck case} of the algorithm is illustrated in \Cref{fig:bottleneck}.
In the original paper, it is a multistep rule. In the first step, one takes~$\mathcal{D} = (a) \lor (\lnot a)$. Its corresponding branching vector is~$\{4, 10\}$, and the branching complexity is of~$1.1120$.
Then in the next step, a \emph{fine structure} will appear, and the branching rule with lower complexity can be applied.
Such multistep branching gives an overall $\gamma = 1.0836$.

\begin{figure}[ht!]
    \centering
    \begin{tikzpicture}[decoration=brace,scale=0.7]
    \fill[green!70!black!50,opacity=0.4] (0,0) ellipse (7.2 and 7.4);
    \fill[white] (0,0) ellipse (6.1 and 6.2);
    \draw[thick] (0,0) ellipse (7.2 and 7.4); 
        \fill[yellow!90!black!80,opacity=0.5] (0,0) ellipse (6.1 and 6.2);
    
        \fill[white!30] (0,0) ellipse (4 and 4);
        
\newcommand{\treelike}{
  \begin{scope}
        \node[circle, fill=white, draw=black, thick, minimum size=7, inner sep=0] at (-0.0, -0.4)(k2) {};
        \node[circle, fill=white, draw=black, thick, minimum size=7, inner sep=0] at (-0.52, -0.7) (k3) {};
        \node[circle, fill=white, draw=black, thick, minimum size=7, inner sep=0] at (0.52, -0.7) (k4) {};
        \node[circle, fill=white, draw=black, thick, minimum size=7, inner sep=0] at (-0.22, -1.124) (k5) {};
        \node[circle, fill=white, draw=black, thick, minimum size=7, inner sep=0] at (-0.82, -1.124) (k6) {};
        \node[circle, fill=white, draw=black, thick, minimum size=7, inner sep=0] at (0.22, -1.124) (k7) {};
        \node[circle, fill=white, draw=black, thick, minimum size=7, inner sep=0] at (0.82, -1.124) (k8) {};
         \node[] at (-0.22, -2.1) (k9) {};
        \node[] at (-0.82, -2.1) (k10) {};
        \node[] at (0.22, -2.1) (k11) {};
        \node[] at (0.82, -2.1) (k12) {};
            \draw[thick] (k2) -- (k3);
    \draw[thick] (k2) -- (k4);
    \draw[thick] (k3) -- (k5);
    \draw[thick] (k3) -- (k6);
    \draw[thick] (k4) -- (k7);
    \draw[thick] (k4) -- (k8);
    \draw[dashed] (k5) -- (k9);
    \draw[dashed] (k6) -- (k10);
    \draw[dashed] (k7) -- (k11);
    \draw[dashed] (k8) -- (k12);
  \end{scope}
}

\newcommand{\branchc}{
  \begin{scope}
  \node[circle, fill=black, draw=black, thick, minimum size=10, inner sep=0] at (0.0, 0.0) (c) {\color{white}{c}};
  \node[circle, fill=black, draw=black, thick, minimum size=10, inner sep=0] at (-1.0, 1.0) (g) {\color{white}{g}};
  \node[circle, fill=black, draw=black, thick, minimum size=10, inner sep=0] at (1.0, 1.0) (h) {\color{white}{h}};
  \node[circle, fill=black, draw=black, thick, minimum size=10, inner sep=0] at (-1.577, 2.0) (o) {\color{white}{o}};
  \node[circle, fill=black, draw=black, thick, minimum size=10, inner sep=0] at (-0.423, 2.0) (r) {\color{white}{r}};
  \node[circle, fill=black, draw=black, thick, minimum size=10, inner sep=0] at (0.423, 2.0) (p) {\color{white}{p}};
  \node[circle, fill=black, draw=black, thick, minimum size=10, inner sep=0] at (1.577, 2.0) (q) {\color{white}{q}};
  \draw[thick] (c) -- (g);
        \draw[thick] (c) -- (h);
        \draw[thick] (g) -- (o);
        \draw[thick] (g) -- (p);
        \draw[thick] (h) -- (r);
         \draw[thick] (h) -- (q);
          \draw[thick] (o) -- (r);
           \draw[thick] (p) -- (q);
     \begin{scope}[xshift=-66,yshift=80,rotate=215]
    \treelike
    \draw[thick] (o) -- (k2);
  \end{scope}
   \begin{scope}[xshift=-25,yshift=100,rotate=-170]
    \treelike
    \draw[thick] (r) -- (k2);
  \end{scope}
  \begin{scope}[xshift=25,yshift=100,rotate=170]
    \treelike
    \draw[thick] (p) -- (k2);
  \end{scope}
  \begin{scope}[xshift=66,yshift=80,rotate=-215]
    \treelike
    \draw[thick] (q) -- (k2);
  \end{scope}
  \end{scope}
  }

  \newcommand{\branchb}{
  \begin{scope}
  \node[circle, fill=black, draw=black, thick, minimum size=10, inner sep=0] at (0.0, 0.0) (c) {\color{white}{b}};
  \node[circle, fill=black, draw=black, thick, minimum size=10, inner sep=0] at (-1.0, 1.0) (g) {\color{white}{e}};
  \node[circle, fill=black, draw=black, thick, minimum size=10, inner sep=0] at (1.0, 1.0) (h) {\color{white}{f}};
  \node[circle, fill=black, draw=black, thick, minimum size=10, inner sep=0] at (-1.577, 2.0) (o) {\color{white}{k}};
  \node[circle, fill=black, draw=black, thick, minimum size=10, inner sep=0] at (-0.423, 2.0) (r) {\color{white}{n}};
  \node[circle, fill=black, draw=black, thick, minimum size=10, inner sep=0] at (0.423, 2.0) (p) {\color{white}{l}};
  \node[circle, fill=black, draw=black, thick, minimum size=10, inner sep=0] at (1.577, 2.0) (q) {\color{white}{m}};
  \draw[thick] (c) -- (g);
        \draw[thick] (c) -- (h);
        \draw[thick] (g) -- (o);
        \draw[thick] (g) -- (p);
        \draw[thick] (h) -- (r);
         \draw[thick] (h) -- (q);
          \draw[thick] (o) -- (r);
           \draw[thick] (p) -- (q);
     \begin{scope}[xshift=-66,yshift=80,rotate=215]
    \treelike
    \draw[thick] (o) -- (k2);
  \end{scope}
   \begin{scope}[xshift=-25,yshift=100,rotate=-170]
    \treelike
    \draw[thick] (r) -- (k2);
  \end{scope}
  \begin{scope}[xshift=25,yshift=100,rotate=170]
    \treelike
    \draw[thick] (p) -- (k2);
  \end{scope}
  \begin{scope}[xshift=66,yshift=80,rotate=-215]
    \treelike
    \draw[thick] (q) -- (k2);
  \end{scope}
  \end{scope}
  }

  \newcommand{\branchd}{
  \begin{scope}
  \node[circle, fill=black, draw=black, thick, minimum size=10, inner sep=0] at (0.0, 0.0) (c) {\color{white}{d}};
  \node[circle, fill=black, draw=black, thick, minimum size=10, inner sep=0] at (-1.0, 1.0) (g) {\color{white}{i}};
  \node[circle, fill=black, draw=black, thick, minimum size=10, inner sep=0] at (1.0, 1.0) (h) {\color{white}{j}};
  \node[circle, fill=black, draw=black, thick, minimum size=10, inner sep=0] at (-1.577, 2.0) (o) {\color{white}{s}};
  \node[circle, fill=black, draw=black, thick, minimum size=10, inner sep=0] at (-0.423, 2.0) (r) {\color{white}{v}};
  \node[circle, fill=black, draw=black, thick, minimum size=10, inner sep=0] at (0.423, 2.0) (p) {\color{white}{t}};
  \node[circle, fill=black, draw=black, thick, minimum size=10, inner sep=0] at (1.577, 2.0) (q) {\color{white}{u}};
  \draw[thick] (c) -- (g);
        \draw[thick] (c) -- (h);
        \draw[thick] (g) -- (o);
        \draw[thick] (g) -- (p);
        \draw[thick] (h) -- (r);
         \draw[thick] (h) -- (q);
          \draw[thick] (o) -- (r);
           \draw[thick] (p) -- (q);
     \begin{scope}[xshift=-66,yshift=80,rotate=215]
    \treelike
    \draw[thick] (o) -- (k2);
  \end{scope}
   \begin{scope}[xshift=-25,yshift=100,rotate=-170]
    \treelike
    \draw[thick] (r) -- (k2);
  \end{scope}
  \begin{scope}[xshift=25,yshift=100,rotate=170]
    \treelike
    \draw[thick] (p) -- (k2);
  \end{scope}
  \begin{scope}[xshift=66,yshift=80,rotate=-215]
    \treelike
    \draw[thick] (q) -- (k2);
  \end{scope}
  \end{scope}
  }

\node[circle, fill=black, draw=black, thick, minimum size=10, inner sep=0] at (0.0, 0.0) (a) {\color{white}{a}};

\begin{scope}[yshift=30]
    \branchc'=
    \draw[thick] (a) -- (c);
\end{scope}

\begin{scope}[xshift=-25.98,yshift=-15,rotate=120]
    \branchb
    \draw[thick] (a) -- (c);
\end{scope}

  \begin{scope}[xshift=25.98,yshift=-15,rotate=-120]
    \branchd
    \draw[thick] (a) -- (c);
\end{scope}

         \textcolor{black}{\node[font=\Large] at (-2.2, 1.4) {$R$};}
        \textcolor{black}{\node[font=\Large] at (-4.2, 3) {$N_3[R]$};}
        \textcolor{black}{\node[font=\Large] at (-5.55,3.6) {$G$};}
    \end{tikzpicture}
    \caption{
        The example of the bottleneck case in Ref.~\cite{xiao2013}. The neighborhood of $R$ is chosen to be tree-like.
        }
    \label{fig:bottleneck}
\end{figure}
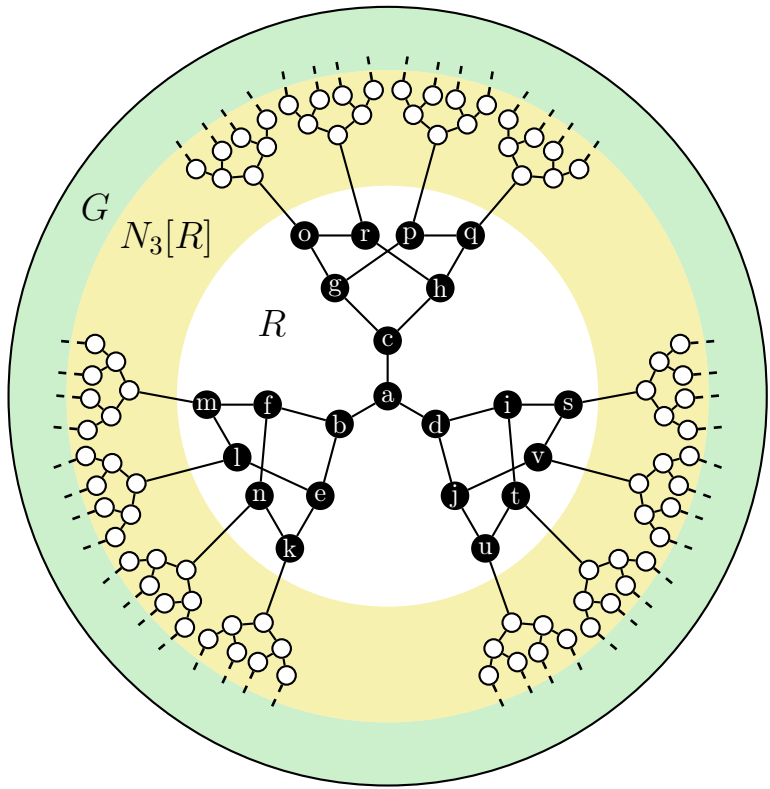

Here, we also consider a special case where $N_3[R]$ is tree-like. 
When applying our method, this structure generates $71$ entries in the reduced-$\alpha$-tensor and $15782$ candidate clauses.
This seemingly large integer programming problem can be solved in a few seconds due to the easy problem structure.
The resulting optimal branching rule on $R$ yields the branching vector $\{10,16,26,26\}$ with a complexity of $\gamma=1.0817$, which is lower than that of the multistep branching.
Although this tree-like environment is considered a hard instance for multistep branching, we admit that the optimal branching does not consider all possible choices of the environment. A rigorous proof of having a lower complexity is left as a future work.

\section{On-the-fly branch-and-reduce}\label{sec:numerical-results}

This section presents a numeric method that utilizes the optimal branching algorithm for solving the MIS problem.
Its performance is benchmarked against existing methods across various graphs, including 3-regular graphs, Erdos-Renyi graphs and geometric graphs.
Additionally, the performance of its LP relaxation is also examined.
Our methods are implemented based on the Julia Programming Language~\cite{julia-2017} and are open-sourced on GitHub~\cite{optimalbranching}. 
A comprehensive technical guide is available in \Cref{sec:technical}.

\subsection{The algorithm}

We improve the branch-and-reduce algorithm by including our optimal branching algorithm to generate the branching rules on-the-fly.
A branch-and-reduce algorithm contains two building blocks: reduction and branching.
Reduction is a graph rewriting process that replaces specific sub-graphs with smaller ones.
The simplest reduction is the d1/d2 reduction that only handles degree 1 and 2 vertices.
Although some reduction rules such as the d1 reduction and the dominance rule that only involve vertex removal can be automatically discovered by our optimal branching algorithm, reduction rules that requires more sophisticated rewriting fail to fit into the branching framework.
Therefore, we still need to borrow the existing reduction rules from the previous methods.
The reduction rules and branching algorithms used in this study are listed in \Cref{tbl:algorithms}.
Two different sets of additional reduction rules are used. They are the Xiao's rules that are used in Ref.~\cite{xiao2013} and the packing rule from Ref.~\cite{akiba2016branch}.
Since reduction does not increase the number of branches, it is a free-to-use resource.
Branching is performed exclusively when no further reduction is feasible.
The choice of vertex set for branching (the function \texttt{select\_subgraph}) is more flexible than the previous methods.
A heuristic vertex selection strategy is used, where we check each vertex $v$ and its neighborhood $N_2[v]$, and branch over the region with the fewest boundary vertices.

\begin{table}[ht!]
    \centering
    \renewcommand{\arraystretch}{2}
    \begin{tabular}{|c|c|c|c|}
    \hline
    \diagbox{\makecell[c]{Reduction\\Rules}}{\makecell[c]{Branching\\Algorithms}}&\makecell[c]{Optimal\\Branching\\(this work)} & \makecell[c]{Xiao 2013~\cite{xiao2013}}& \makecell[c]{Akiba 2015~\cite{akiba2016branch}} \\
    \hline
     \makecell[c]{d1/d2 reduction} & \texttt{ob} & \rule{1cm}{0.4pt} & \texttt{akiba2015}\\
    \hline
    \makecell[c]{d1/d2 reduction\\Xiao's rules~\cite{xiao2013}} & \texttt{ob+xiao} & \texttt{xiao2013}  & \rule{1cm}{0.4pt}\\
    \hline
    \makecell[c]{d1/d2 reduction\\Xiao's rules\\packing rule~\cite{akiba2016branch}} & \rule{1cm}{0.4pt} & \rule{1cm}{0.4pt}  & \makecell[c]{\texttt{akiba2015+} \\ \texttt{xiao\&packing}}\\
    \hline
    \end{tabular}
    \caption{
        The various branch-and-reduce algorithms used in this study.
        Different rows correspond to different reduction rules and different columns correspond to different branching algorithms.
        \texttt{ob+xiao} excludes the packing rule since it can be automatically discovered by our optimal branching algorithm.
        }
    \label{tbl:algorithms}
\end{table}

\subsection{Comparison with the existing methods}
As shown in \Cref{tbl:algorithms}, two state-of-the-art methods are compared with our method.
The \texttt{xiao2013}~\cite{xiao2013} is tailored for 3-regular graphs, which possesses the lowest theoretical computational complexity of $O(1.0836^n)$.
As we will show later, its practical performance is much better than the theoretical complexity.
The \texttt{akiba2015} and \texttt{akiba2015+xiao\&packing}~\cite{akiba2016branch} are from the winning solver of the minimum vertex cover (the same as MIS) problem in the exact track at the PACE 2019 challenge~\cite{hespe2020wegotyoucovered}.
We use the total number of branches as our metric, which reflects how much the branching algorithm can prune the search space.
The performance is tested over 4 types of graphs: 3-regular graphs, Erdos-Renyi graphs, King's subgraphs~\cite{Ebadi2022}, and grid graphs. For the Erdos-Renyi graphs, we set the average degree to $d=3$, while for the King's sub-graphs and grid graphs, we select a filling rate of $f=0.8$. 
The results of this comparison are presented in \Cref{fig:branching_comparison}, where each data point represents the geometric average of results from 1,000 runs conducted on randomly generated graphs.

We fitted the average branching factor $\gamma$ with the above data and present the results in \Cref{tbl:branching_comparison}.
For 3-regular graphs, we find the algorithms \texttt{ob} and \texttt{ob+xiao} reduce the average complexity to $O(1.0457^n)$ and $O(1.0441^n)$, respectively, outperforming the \texttt{xiao2013}'s $O(1.0487^n)$.
Furthermore, \texttt{ob+xiao} consistently produces fewer branches than \texttt{xiao2013} across all tested problem sizes. 
In other graphs, \texttt{ob} demonstrates superior performance compared to \texttt{akiba} when using only d1/d2 reduction in both cases.
Additionally, despite using fewer reduction rules, \texttt{ob+xiao} achieves a performance comparable to that of \texttt{akiba2015+xiao\&packing}.
These results affirm the effectiveness and generality of the optimal branching rule generation algorithm across a variety of graphs, regardless of whether they are bounded or unbounded, geometric or non-geometric.
While the above study is based on the average case, the worst-case performances are presented in \Cref{sec:worst_case}, revealing conclusions that align closely with those observed in the average case.

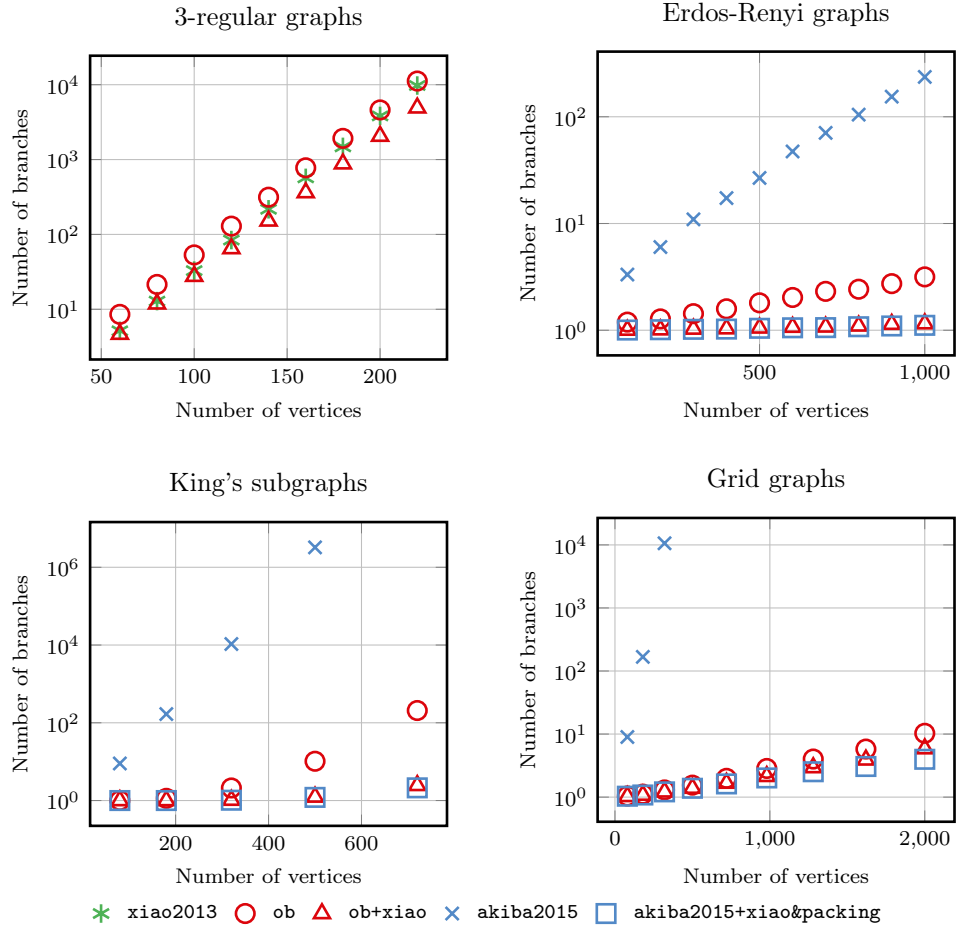
\begin{figure*}[th]
    \begin{subfigure}[b]{0.48\textwidth}
      \centering
    \input{images/3rr_mean.tex}
      \label{fig:3rr_mean}
    \end{subfigure}
    \hfill
    \begin{subfigure}[b]{0.48\textwidth}
      \centering
    \input{images/er_mean.tex}
      \label{fig:er_mean}
    \end{subfigure}
    \vskip\baselineskip
    \begin{subfigure}[b]{0.48\textwidth}
      \centering
    \input{images/ksg_mean.tex}
      \label{fig:ksg_mean}
    \end{subfigure}
    \hfill
    \begin{subfigure}[b]{0.48\textwidth}
      \centering
    \input{images/sq_mean.tex}
      \label{fig:sq_mean}
    \end{subfigure}
  
    \begin{subfigure}{\textwidth}
      \centering
      \input{legend_compare.tex}
      \label{fig:legend_compare}
    \end{subfigure}
  
    \caption{
        The average number of branches produced by various algorithms on different graphs as functions of the graph size.
    }%
  
    \label{fig:branching_comparison}
  \end{figure*}

  \begin{table}[ht!]
    \centering
    \begin{tabular}{c c c c c c}
    \hline
         & \texttt{ob} & \makecell[c]{\texttt{ob+}\\ \texttt{xiao}} & \texttt{xiao2013} & \texttt{akiba2015} & \makecell[c]{\texttt{akiba2015+} \\ \texttt{xiao\&packing}} \\
        \hline
        {3-regular graphs} & 1.0457 & 1.0441 & 1.0487 & - & - \\
        {Erdos-Renyi graphs} & 1.0011 & 1.0002 & - & 1.0044 & 1.0001 \\
        {King's sub-graphs} & 1.0116 & 1.0022 & - & 1.0313 & 1.0019 \\
        {Grid graphs} & 1.0012 & 1.0009 & - & 1.0294 & 1.0007 \\
        \hline
    \end{tabular}
    \caption{
        The average branching factor $\gamma$ for the branch-and-reduce algorithms in \Cref{tbl:algorithms} on different graphs.
        These results are obtained by fitting the data \Cref{fig:branching_comparison}.
        }
    \label{tbl:branching_comparison}
\end{table}

The high performance of on-the-fly branching also comes from the flexible vertex selection strategy.
Our selection strategy simply selects the second order neighborhood with the fewest boundary vertices, which is a larger region with a small boundary.
This implies more internal constraints, making it more likely to yield effective branching rules. 
Moreover, there is a room to further improve the vertex selection strategy, e.g. we can select a long chain as the branching region instead of limiting to the second order neighborhood.
In the past, we do not have this freedom since we need to detect the mirrors~\cite{Fomin2013}, satellites~\cite{kneis2009fine}, and other structures that only live in the second order neighborhood.

\subsection{LP relaxation}

While the integer programming solver is already efficient enough for us to generate the rules on-the-fly,
by relaxing the integer programming to linear programming, we can further speed up the generation of the branching rules.
While linear programming is convex and can be solved in polynomial time, it does not guarantee the optimal solution for the WMSC problem, potentially resulting in a higher branching factor.
The relaxation of \texttt{ob} and \texttt{ob+xiao} are denoted as \texttt{ob\_relax} and \texttt{ob\_relax+xiao}, respectively, and the results are presented in \Cref{fig:branching_comparison_iplp}.
It is confirmed that the linear programming relaxation results in an increased number of branches, however, the difference is small across all tested problems. 

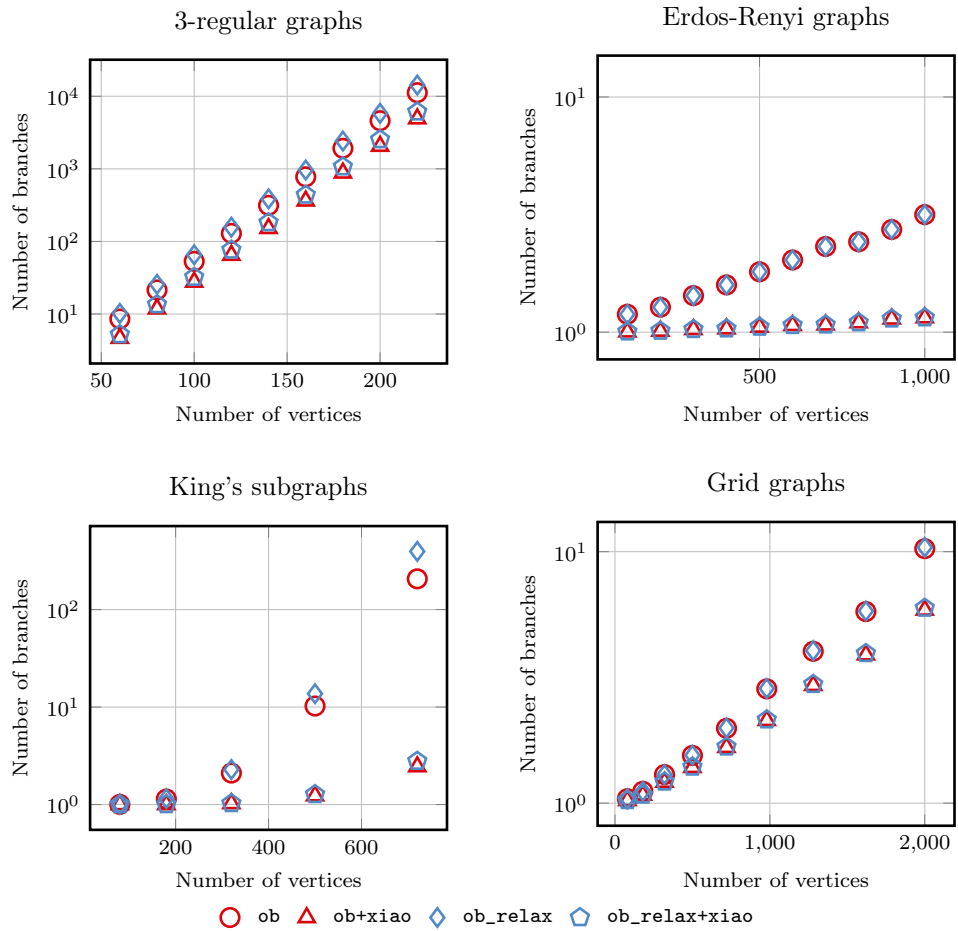
\begin{figure*}[th]
    \begin{subfigure}[b]{0.48\textwidth}
        \centering
      \input{images/3rr_mean_ip_lp.tex}
        \label{fig:3rr_mean_ip_lp}
      \end{subfigure}
      \hfill
      \begin{subfigure}[b]{0.48\textwidth}
        \centering
      \input{images/er_mean_ip_lp.tex}
        \label{fig:er_mean_ip_lp}
      \end{subfigure}
      \vskip\baselineskip
      \begin{subfigure}[b]{0.48\textwidth}
        \centering
      \input{images/ksg_mean_ip_lp.tex}
        \label{fig:ksg_mean_ip_lp}
      \end{subfigure}
      \hfill
      \begin{subfigure}[b]{0.48\textwidth}
        \centering
      \input{images/sq_mean_ip_lp.tex}
        \label{fig:sq_mean_ip_lp}
      \end{subfigure}
  
    \begin{subfigure}{\textwidth}
      \centering
      \input{legend_compare_iplp.tex}
    \end{subfigure}
  
    \caption{
        The average number of branches produced by \texttt{ob}, \texttt{ob\_relax}, \texttt{ob+xiao}, and \texttt{ob\_relax+xiao} on different graphs as functions of the graph size.
    }%
  
    \label{fig:branching_comparison_iplp}
  \end{figure*}


\section{Conclusion}\label{sec:conclusion}
In this paper, we present a framework to automatically generate provably optimal branching rules for constraint satisfaction problems, and implement it to solve the maximum independent set problems better.
We also show that this framework can generate better branching rules than the existing ones, which could be used to improve the complexity bound of existing branch-and-bound algorithms.
We also implement the on-the-fly optimal branching algorithm. Numerical results on 3-regular graphs show an advantage compared with state-of-the-art algorithms, even though no predefined branching rule is used.

Although the proposed algorithm can generate the optimal branching rule on sub-graphs with tens of nodes, it is not the end of improving the branch-and-bound algorithm. The vertex selection strategy, the measure function, the way to truncated candidate clauses, and pre-processing with graph rewriting are all potential research directions to further improve the performance of the branch-and-bound algorithm.
The source code under the MIT license is available at the GitHub repository~\cite{optimalbranching}, and a technical guide about how to use the code is available at~\Cref{sec:technical}.

In the future, we expect this framework can be extended to a variety of combinatorial optimization challenges, including the Vertex Cover problem~\cite{xiao2010note}, the Max-SAT problem~\cite{chen2004improved}, and the Traveling Salesman Problem~\cite{eppstein2007traveling}. More importantly, it may also be used in propositional proof and resolution~\cite{urquhart1995complexity,alekhnovich2005lower} to improve the performance of logic reasoning.

\section*{Acknowledgments}
The authors thank Huanhai Zhou, Kaiwen Jin, and Zisong Shen for helpful discussions. 
We acknowledge the use of AI tools like Grammarly and ChatGPT for sentence rephrasing and grammar checks.

\bibliographystyle{siamplain}
\bibliography{references}

\appendix

\section{$\alpha$-tensor and reduced $\alpha$-tensor}\label{sec:alpha-tensor}

The finite-valued entries of $\alpha$-tensor correspond to permissible configurations and they all have the potential to constitute a segment of an MIS configuration. However, since our objective is to determine the MIS size rather than enumerating all MIS configurations, the entries of the $\alpha$-tensor can be further reduced. This pruning is under the basic criterion that if all outcomes under one branch can find better counterparts under another branch, then this branch can be safely discarded during the search process. In the context of the MIS size problem, when branching on the subgraph $R$ of a given graph $G$, consider two configurations $s_1$ and $s_2$ on $R$ and let the set $S$ consist of all MIS configurations on $G$. If for any element $s$ in $S$ satisfying that its segment restricted to $R$ is equal to $s_1$, the element obtained by replacing the segment of $s_1$ in $s$ with $s_2$ still belongs to $S$, then it implies that the MIS size outcomes under the branch corresponding to $s_1$ are a subset of those under branch $s_2$. Consequently, the $\alpha$-tensor entry corresponding to $s_1$ can be safely reduced without affecting the outcome. 
\subsection{Pruning irrelevant entries}\label{sec:irrelevant_entries}
To implement this pruning principle, we first utilize only the permissible configurations and the corresponding current MIS size on $R$ to find such $s_1$-$s_2$ pairs. However, with only local information available, we cannot make assumptions about the entire graph $G$ during branching. Consequently, the pruning condition is leveraged such that $s_1$ and $s_2$ must ensure the containment relationship of branch outcomes for all graphs $G$ that include a subgraph isomorphic to $R$. This defines a partial order on the configuration space on $R$, which is formalized with the following two definitions.
\begin{definition}[less restrictive relation]\label{def:less-restrictive}
Consider two bitstrings of equal length $n$, $\boldsymbol{s}\in \{0,1\}^n$ and $\boldsymbol{t}\in \{0,1\}^n$. We say that $\boldsymbol{s}$ is less restrictive than ${\boldsymbol{t}}$ if $s_i\leq t_i$ for all $i\in\{1,\dots,n\}$. We denote this by ${\boldsymbol{s}} \prec {\boldsymbol{t}}$.
\end{definition}
Clearly, the least restrictive boundary configuration is the one containing all zeros. For any other boundary configuration, one can always find less restrictive ones by flipping some number of ones to zeros.
\begin{definition}[irrelevant boundary configuration]\label{def:irrelevant-boundary}
Let $R$ be a subgraph of a graph, $\partial R$ its boundary, and $\valpha(R)$ its $\alpha$-tensor. A boundary configuration $\boldsymbol{t}$ is irrelevant if there exists another boundary configuration $\boldsymbol{s}$ such that $\boldsymbol{s} \prec \boldsymbol{t}$ and $\valpha(R)_{\boldsymbol{s}}\geq \valpha(R)_{\boldsymbol{t}}$, or if $\valpha(R)_{\boldsymbol{t}}=-\infty$. A boundary configuration is called relevant if it is not irrelevant.
\end{definition}
Under these definitions, since all independent sets emerging from the $t$ branch can find counterparts in the  $s$ branch that maintain the configurations in $G \backslash R$ but larger in size, the $t$ branch can be safely discarded in the presence of the $s$ branch. We call that the boundary configuration $t$ can be reduced by $s$, which defines a partial order on the effective boundary configurations, and we only need to continue searching in the branches corresponding to the maximal elements of this partially ordered set.
Therefore, we can define the reduced $\alpha$-tensor by setting all entries in the $\alpha$-tensor that correspond to irrelevant boundary vertex configurations to $-\infty$.
\begin{example}[Reduced irrelevant entries to obtain reduced $\alpha$-tensor]
    \centering
    \begin{tabular}{|c|c|c|c|}
    \hline
        $\psigma$ & $\valpha(R)_\psigma$ & $\rvalpha(R)_\psigma$ & reduced by\\
        \hline
        000 & 1 & 1 & -  \\
        001 & 2 & 2 & -\\
        010 & 2 & 2 & - \\
        011 & 2 & $-\infty$ & 010\\
        100 & 1 & $-\infty$ & 000\\
        101 & 2 & $-\infty$ & 001  \\
        110 & 2 & $-\infty$ & 010 \\
        111 & 3 & 3 & -  \\
        
        \hline
    \end{tabular}
    \captionof{table}{The finite-valued entries in $\alpha$-tensor and the corresponding reduced $\alpha$-tensor entries for the sub-graph $R$ in \Cref{fig:branching}. Each row corresponds to a local MIS size associated with the boundary-vertex configuration $\bs_{abc}$. 
    Since no other configurations can reduce configurations $000$, $001$, $010$, and $111$, they are maximal elements in the partial order set, i.e., they are the relevant configurations, whose values on $\rvalpha(R)$ are equal to those on $\valpha(R)$. At the same time, since configurations $011$, $100$, $101$, and $110$ can be reduced by the relevant configurations, their corresponding entries in $\rvalpha(R)_\psigma$ are set to $-\infty$.}
    \label{tbl:alpha-tensor-R}    

\end{example}

\subsection{Enhanced pruning via incorporating nearest neighbor information}\label{sec:incorporating_env_info}
It has been demonstrated that the entries in a reduced $\alpha$-tensor can not be further reduced if no environmental information is provided, i.e. the further removal of any finite-valued elements in the reduced $\alpha$-tensors may reduce the MIS size of the host graph. 
However, with a pruning scheme that utilizes environmental information from $R$, we can further reduce the number of entries in the reduced $\alpha$-tensor.

As the amount and precision of environmental information increase, so does the computational cost at each branching step. To address this, we have chosen a strategy that minimizes costs while still considering the environment, concentrating exclusively on the nearest neighbors of boundary vertices and disregarding the connectivity structure among these neighbors.
Let $\alpha(G)$ denotes the MIS size on graph $G$ and $\alpha_{\bs_{V}=\bx}(G)$ denotes the MIS size when vertices $V \subseteq V(G)$ are fixed with configuration $\bs_{V}=\bx$. 
Thus, our criterion for pruning can be traced back to the inequality:
\begin{equation}
    \alpha(G_1 \cup G_2) \leq \alpha(G_1)+\alpha(G_2),
\end{equation}
where $G_1 \cup G_2$ denotes the induced subgraph of $G$ from $V(G_1) \cup V(G_2)$. 
Given a boundary configuration $\bs$ on subgraph $R$, we denote $G \setminus (N(T(\bs)) \cup V(R))$ as $G_{\text{left}}(\bs)$, where $T(\bs)$ is the set of boundary vertices with true assignment in $\bs$.
Then the independent set consistent with $\bs$ has a size that is upper bounded by the one consistent with another boundary configuration $\bt$:
\begin{equation}
    \begin{split}
        \alpha_{\psigma=\bs}(G) =& \alpha_{\psigma=\bs}(R)+\alpha_{\psigma=\bs}(G \backslash R)\\
        =&\valpha(R)_\bs + \alpha(G_{\text{left}}(\bs))\\
        \leq & \valpha(R)_\bs + \alpha(G_{\text{left}}(\bs)\backslash G_{\text{left}}(\bt)) + \alpha(G_{\text{left}}(\bs)\cap G_{\text{left}}(\bt))\\
        \leq & \valpha(R)_\bs + \alpha(G_{\text{left}}(\bs)\backslash G_{\text{left}}(\bt)) + \alpha(G_{\text{left}}(\bt))\\
    \end{split}
\end{equation}

If the MIS size within $R$ and the nearest-neighbor graph structure upon which $G_{\text{left}}(\bs)\backslash G_{\text{left}}(\bt)$ depends satisfy that
\begin{equation}\label{eq:env_reduced_relation}
\valpha(R)_s+ \alpha(G_{\text{left}}(s)\backslash G_{\text{left}}(t)) \leq \valpha(R)_t, 
\end{equation}
then we have:
\begin{equation}
    \begin{split}
        \alpha_{\psigma=\bs}(G) \leq& \valpha(R)_\bt + \alpha(G_{\text{left}}(\bt))\\
        =&\alpha_{\psigma=\bt}(R)+\alpha_{\psigma=\bt}(G \backslash R)\\
        =&\alpha_{\psigma=\bt}(G)\\
    \end{split}
\end{equation}
which means that the maximum MIS size under the $\bs$ branch is bounded by the $\bt$ branch, thus the $\bs$ branch can be safely discarded during the search process. 
Thus, \Cref{eq:env_reduced_relation} defines a partial order on the boundary configurations. 
By setting entries corresponding to elements outside the maximal elements of this partial order set to $-\infty$, we can further simplify the reduced $\alpha$-tensor.

As $G_{\text{left}}(\bs)\backslash G_{\text{left}}(\bt)$ only involves the union and difference between the nearest neighbors of boundary vertices, solving for its MIS size is a local and manageable task, which is a minor component compared to the dominant computational cost of the whole algorithm. Consequently, without altering the order of computational complexity, we enhance the pruning strategy by considering information from the neighboring environment.

\begin{example}[Fine-grained pruning via environment]
    \centering
    \begin{tabular}{|c|c|c|c|}
    \hline
        $\psigma$ & $\rvalpha(R)_\psigma$ & $\rvalpha^{\text{env}}(R)_\psigma$ & reduced by\\
        \hline
        000 & 1 & $-\infty$ & 001  \\
        001 & 2 & 2 & -\\
        010 & 2 & 2 & - \\
        111 & 3 & 3 & -  \\
        
        \hline
    \end{tabular}
    \captionof{table}{The finite-valued entries in the reduced $\alpha$-tensor $\rvalpha(R)_\psigma$ after reducing irrelevant entries and the corresponding entries of the futher-reduced $\alpha$-tensor $\rvalpha^{\text{env}}(R)_\psigma$ for the sub-graph $R$ in \Cref{fig:branching}, with an extra assumption that $|N(c)\backslash V(R)|=1$. Each row corresponds to a local MIS size associated with the boundary-vertex configuration $\bs_{abc}$. 
    Since $\valpha(R)_{000}+ \alpha(G_{\text{left}}(000)\backslash G_{\text{left}}(001)) = 1+|N(c)\backslash V(R)| = 2 = \valpha(R)_{001}$, $000$ can be reduced by $001$ and consequently, $\rvalpha^{\text{env}}(R)_{000}$ is set to $-\infty$.}   

\end{example}



\begin{figure}[htbp!]
    \centering
    \includegraphics[width=0.6\textwidth]{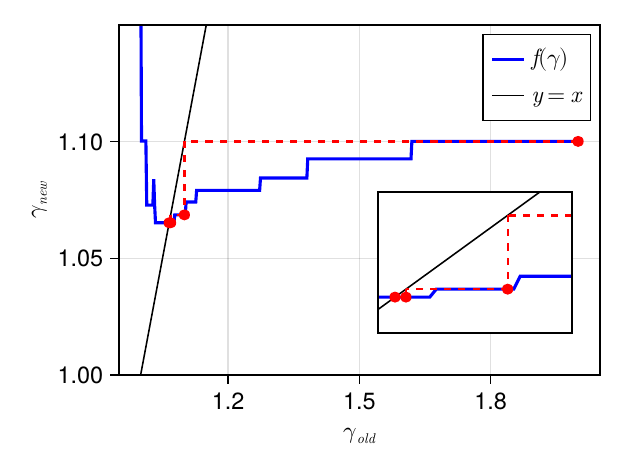}
    \caption{
        An example of the fixed point iteration of $\gamma$ in \Cref{alg:opt_branching}, the inset subplot shows the iteration near the fixed point.
        The blue line represents the value of~$f(\gamma)$, the black line represents~$y = x$, and the red dashed lines and points represent the iteration steps.
        }
    \label{fig:fixed_point_gamma}
\end{figure}

\section{Fixed Point Iteration}\label{sec:fixed_point}

In this section, we provide more details about the fixed point iteration used in~\Cref{alg:opt_branching}.
Let $f$ be the function representing the process described in the while loop of~\Cref{alg:opt_branching}, which takes a~$\gamma_{old}$ as input to generate a new~$\gamma_{new}$. 
Then the following theorems hold.

\begin{theorem}
    The function~$f$ has only one fixed point~$\gamma_0$, i.e.,~$f(\gamma_0) = \gamma_0$, where $\gamma_0$ is the solution of the WMSC problem defined in~\Cref{eq:gamma}.
\end{theorem}

\begin{proof}
    First we define the following function:
    \begin{equation}
        g(\gamma, \V{x}) = \sum_{i = 1}^{|\mathcal{C}|} \gamma^{-\Delta \rho(c_i)} x_i\;,
    \end{equation} 
    which is is monotonically decreasing to~$\gamma$ with $\V{x}$ fixed, since~$\Delta \rho$ are non-negative.
    Recall WMSC problem defined in Eq.~\eqref{eq:gamma}, let its solution be~$(\gamma_{\text{wmsc}}, \V{x}_{\text{wmsc}})$, 
    so that $\forall \V{x}^\prime \neq \V{x}_{\text{wmsc}}$ satisfies the constraints in Eq.~\eqref{eq:integer_programming}, $g(\gamma_{\text{wmsc}}, \V{x}^\prime) \geq 1$ (if not, a smaller~$\gamma$ can be found).
    Thus, in the iteration process taking $\gamma_{old} = \gamma_{\text{wmsc}}$ as input, the output is given by $\gamma_{new} = \gamma_{\text{wmsc}}$, i.e., $\gamma_{\text{wmsc}}$ is the fixed point of the iteration function~$f$.

    Let~$\gamma_1 > \gamma_0$, then in the integer programming progress, the function~$g$ is minimized over~$\V{x}$ so that resulting in~$g(\gamma_1, \V{x}') < g(f(\gamma_1), \V{x}') = 1$, which implies~$f(\gamma_1) < \gamma_1$.
    Thus,~$\forall \gamma > \gamma_0$ can not be a fixed point of~$f$, i.e.,~$\gamma_0$ is the unique fixed point of~$f$.
\end{proof}

\begin{theorem}
    For any~$\gamma > \gamma_0$, the fixed point iteration in~\Cref{alg:opt_branching} converges to~$\gamma_0$.
\end{theorem}

\begin{proof}
    As been proved in the previous theorem, the function~$f$ has only one fixed point~$\gamma_0$, and for any~$\gamma > \gamma_0$, the sequence~$\gamma, f(\gamma), f(f(\gamma)), \ldots$ is monotonically decreasing and bounded below by~$\gamma_0$.
    Hence, the sequence converges to~$\gamma_0$.

    What's more, since for any $t \geq 1$, $f^t(\gamma) \in \Gamma=\{1 \textless \gamma \leq 2 \mid \exists \V{x}_0$  satisfying the constraints in \Cref{eq:integer_programming},  
    $\text{s.t. } g(\gamma,\V{x}_0)=1\}$ which is a set of finite size, there are at most $|\Gamma|$ possible $f^t(\gamma)$ during the iteration process. Therefore, $\gamma_0$ can be precisely reached and the number of iteration steps needed is bounded by the number of solutions to the corresponding set covering the problem.
\end{proof}

An example is shown in \Cref{fig:fixed_point_gamma}, which describes the iteration of~$\gamma$ during solving the MIS of a randomly generated 3-regular graph.
It is shown that the iteration converges to the fixed point~$\gamma_0$ rapidly in a few steps as expected.

\section{Worst case of the branching algorithm on random graphs} \label{sec:worst_case}

The maximum number of branches generated by the branching algorithm applied to the random graphs depicted in \Cref{fig:branching_comparison} is illustrated in \Cref{fig:branching_comparison_worst}. The results for three-regular graphs closely resemble those of the average case, indicating that the branching complexity is not significantly affected by the graph structure in this scenario. For unbounded-degree graphs, similar observations can be made as shown in \Cref{fig:branching_comparison}. Notably, \texttt{ob} continues to outperform \texttt{akiba2015} while employing the same reduction rules, and \texttt{ob+xiao} demonstrates performance comparable to that of \texttt{akiba2015+xiao\&packing}.

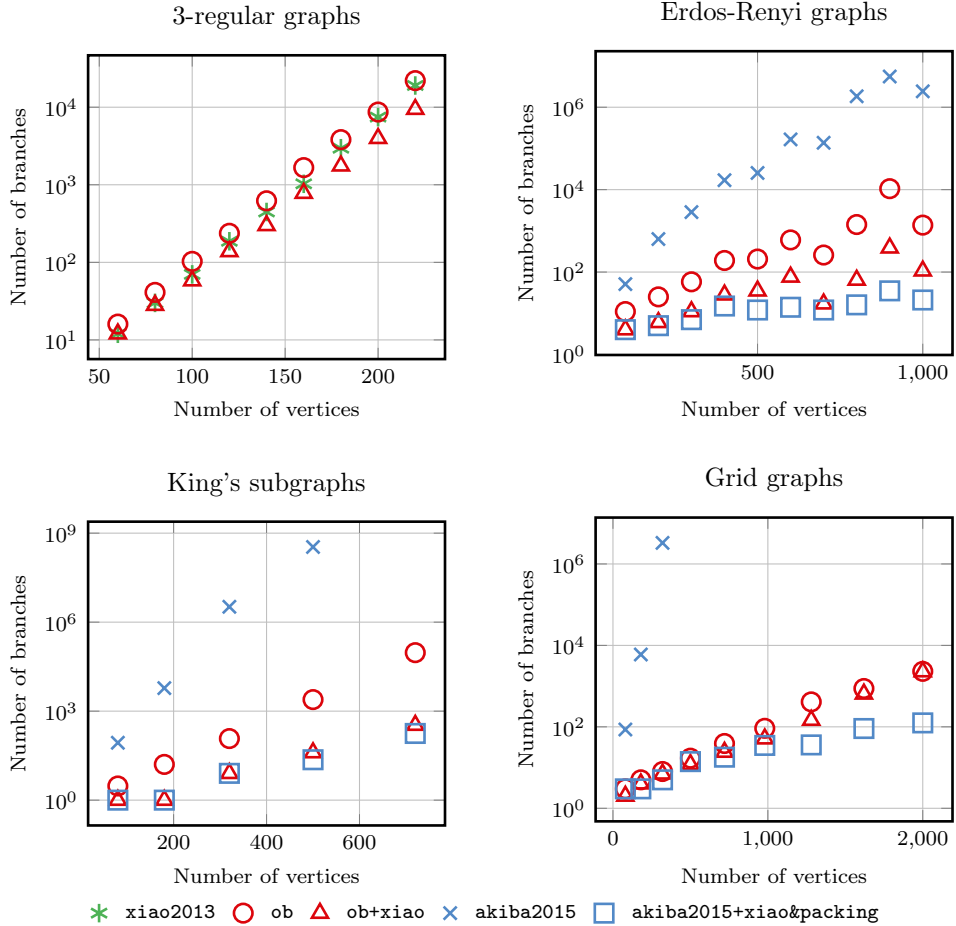
\begin{figure*}[ht!]
    \begin{subfigure}[b]{0.48\textwidth}
      \centering
    \input{images/3rr_worst.tex}
      \label{fig:3rr_mean_worst}
    \end{subfigure}
    \hfill
    \begin{subfigure}[b]{0.48\textwidth}
      \centering
    \input{images/er_worst.tex}
      \label{fig:er_mean_worst}
    \end{subfigure}
    \vskip\baselineskip
    \begin{subfigure}[b]{0.48\textwidth}
      \centering
    \input{images/ksg_worst.tex}
      \label{fig:ksg_mean_worst}
    \end{subfigure}
    \hfill
    \begin{subfigure}[b]{0.48\textwidth}
      \centering
    \input{images/sq_worst.tex}
      \label{fig:sq_mean_worst}
    \end{subfigure}
  
    \begin{subfigure}{\textwidth}
      \centering
      \input{legend_compare.tex}
      \label{fig:legend_compare_worst}
    \end{subfigure}
  
    \caption{
        The largest number of branches produced by various algorithms on different graphs as functions of the graph size.
    }%
  
    \label{fig:branching_comparison_worst}
  \end{figure*}

\section{Technical guides}\label{sec:technical}

This appendix covers some technical guides for efficiency, which is mainly about the open-source package \href{https://github.com/ArrogantGao/OptimalBranching.jl}{OptimalBranching}~\cite{optimalbranching} implementing the algorithms in this paper.
It is built on top of multiple open source packages in the Julia ecosystem: a) \href{https://github.com/QuEraComputing/GenericTensorNetworks.jl}{GenericTensorNetworks}~\cite{Liu2023} is a package for solving the solution space properties (including the sizes, the counting, the enumeration and the sampling of solutions) of a constraint satisfaction problem;
b) \href{https://github.com/jump-dev/JuMP.jl}{JuMP}~\cite{Lubin2023} is one of the most high-performance open-source packages for mathematical optimization. It interfaces multiple backends, the \href{https://github.com/jump-dev/HiGHS.jl}{HiGHS}~\cite{huangfu2018parallelizing} backend is used in this work for solving the linear programming (LP) and \href{https://github.com/scipopt/SCIP.jl}{SCIP}~\cite{Achterberg2009} is used for solving the mixed-integer programming (MIP) problems;
c) \href{https://github.com/JuliaGraphs/Graphs.jl}{Graphs}~\cite{Graphs2021} is a foundational package for graph manipulation in the Julia community.

The \texttt{OptimalBranching} is a meta package that consists of two component packages: \texttt{OptimalBranchingCore} and \texttt{OptimalBranchingMIS}.
The \texttt{OptimalBranchingCore} contains the core functionalities of the optimal branching algorithm and is designed for the sake of modularity and extensibility.
The \texttt{OptimalBranchingMIS} is a specific application of the optimal branching algorithm, which is used to solve the maximum independent set problem.
The \texttt{OptimalBranching} package can be used directly in a Julia REPL as follows:
\begin{lstlisting}
julia> using OptimalBranching, Graphs

julia> graph = smallgraph(:tutte)
{46, 69} undirected simple Int64 graph

julia> mis_branch_count(graph)
(19, 2)
\end{lstlisting}
Here the main function \texttt{mis\_branch\_count} takes a graph as the input and returns a tuple of the maximum independent set size and the number of branches generated by the optimal branching strategy.
In this example, the maximum independent set size of the Tutte graph is 19, and the optimal branching strategy only generates 2 branches in the branching tree.
The first execution of this function will be a bit slow due to Julia's just-in-time compilation.
After that, the subsequent runs will be faster.
The algorithm in default is corresponding to the \texttt{ob\_mis2} in the main text.

The component package \texttt{OptimalBranchingCore} contains the core functionalities of the optimal branching algorithm and is designed for the sake of modularity and extensibility.
The following code snippet shows an example of how to use this module to generate the optimal branching rule.
\begin{lstlisting}
julia> using OptimalBranchingCore, OptimalBranchingCore.BitBasis

julia> tbl = BranchingTable(5, [
        [[0, 0, 0, 0, 1], [0, 0, 0, 1, 0]],
        [[0, 0, 1, 0, 1]],
        [[0, 1, 0, 1, 0]],
        [[1, 1, 1, 0, 0]]])
BranchingTable{LongLongUInt{1}}
10000, 01000
10100
01010
00111

julia> candidates = collect(OptimalBranchingCore.candidate_clauses(tbl))
14-element Vector{Clause{LongLongUInt{1}}}:
 Clause{LongLongUInt{1}}: ¬#5
 Clause{LongLongUInt{1}}: ¬#1 ∧ ¬#2 ∧ #3 ∧ ¬#4 ∧ #5
 Clause{LongLongUInt{1}}: #3 ∧ ¬#4
 Clause{LongLongUInt{1}}: ¬#1 ∧ ¬#3 ∧ #4 ∧ ¬#5
 Clause{LongLongUInt{1}}: ¬#1
 Clause{LongLongUInt{1}}: ¬#1 ∧ #2 ∧ ¬#3 ∧ #4 ∧ ¬#5
 Clause{LongLongUInt{1}}: ¬#1 ∧ ¬#2 ∧ ¬#3 ∧ ¬#4 ∧ #5
 Clause{LongLongUInt{1}}: ¬#4
 Clause{LongLongUInt{1}}: #1 ∧ #2 ∧ #3 ∧ ¬#4 ∧ ¬#5
 Clause{LongLongUInt{1}}: ¬#1 ∧ ¬#2 ∧ ¬#4 ∧ #5
 Clause{LongLongUInt{1}}: ¬#1 ∧ ¬#3
 Clause{LongLongUInt{1}}: ¬#1 ∧ ¬#2
 Clause{LongLongUInt{1}}: #2 ∧ ¬#5
 Clause{LongLongUInt{1}}: ¬#1 ∧ ¬#2 ∧ ¬#3 ∧ #4 ∧ ¬#5

julia> Δρ = [length(literals(sc)) for sc in candidates]; println(Δρ)
[1, 5, 2, 4, 1, 5, 5, 1, 5, 4, 2, 2, 2, 5]

julia> res_ip = OptimalBranchingCore.minimize_γ(tbl, candidates, Δρ, IPSolver())
OptimalBranchingResult{LongLongUInt{1}, Int64}:
 selected_ids: [4, 2, 9]
 optimal_rule: DNF{LongLongUInt{1}}: (¬#1 ∧ ¬#3 ∧ #4 ∧ ¬#5) ∨ (¬#1 ∧ ¬#2 ∧ #3 ∧ ¬#4 ∧ #5) ∨ (#1 ∧ #2 ∧ #3 ∧ ¬#4 ∧ ¬#5)
 branching_vector: [4, 5, 5]
 γ: 1.2671683045421243
\end{lstlisting}

\begin{itemize}
    \item \textbf{Line 3:} We first setup the problem by constructing a branching table \texttt{tbl}, which is a table of the boundary-grouped configurations from the example in \Cref{fig:branching}. To goal is to design a boolean expression in the DNF form, which is satisfied by at least one element from each row in this table. Meanwhile, we require that the branching complexity $\gamma$ is minimized.
    \item \textbf{Line 14:} We generate all possible clauses \texttt{candidates} from the branching table using the \Cref{alg:candidate_clauses}. Here, the total number of candidate clauses is $14$, which is much smaller than $3^5 = 243$.
    The discarded clauses are those that can not form optimal branching rules. They do not use the maximum number of literals to cover the same set of rows in \texttt{tbl}.
    \item \textbf{Line 31:} The associated problem size reduction values $\Delta \rho$ of the clauses are computed. Here, we assume that the size reduction value is only related to the number of literals in the clause for simplicity. In practice, it is also related to the measure $\rho$, and the environment of the chosen sub-graph.
    \item \textbf{Line 34:} With the branching table \texttt{tbl}, the candidate clauses \texttt{candidates}, and the associated problem size reduction values $\Delta \rho$, we minimize the $\gamma$ value (\Cref{eq:gamma}) by iteratively solving the WMSC problem (\Cref{alg:opt_branching}) with the integer programming backend \texttt{IPSolver}.
    The resulting $\gamma$ value is $\sim 1.2672$, the associated branching rule (the 2-nd, 4-th, and 9-th clauses in \texttt{candidates}) is
    \begin{equation}
        \begin{split}
       \mathcal{D} = (&\texttt{\#1} \land \texttt{\#2} \land \texttt{\#3} \land \neg \texttt{\#4} \land \neg \texttt{\#5}) \lor \\
       &(\neg \texttt{\#1} \land \neg \texttt{\#3} \land \texttt{\#4} \land \neg \texttt{\#5}) \lor \\
       &(\neg \texttt{\#1} \land \neg \texttt{\#2} \land \texttt{\#3} \land \neg \texttt{\#4} \land \texttt{\#5}),
        \end{split}
    \end{equation}
    and the associated branching vector is $(4, 5, 5)$.
\end{itemize}

\end{document}

%% file: ex_shared.tex
\usepackage{lipsum}
\usepackage{amsfonts}
\usepackage{graphicx}
\usepackage{epstopdf}
\usepackage{algorithmic}
\usepackage[algo2e, ruled, vlined]{algorithm2e}
\ifpdf
  \DeclareGraphicsExtensions{.eps,.pdf,.png,.jpg}
\else
  \DeclareGraphicsExtensions{.eps}
\fi

\usepackage{amssymb}
\usepackage{jlcode}

\usepackage{diagbox}

\usepackage{listings}
\newsiamremark{remark}{Remark}
\newsiamremark{hypothesis}{Hypothesis}
\crefname{hypothesis}{Hypothesis}{Hypotheses}
\newsiamthm{claim}{Claim}

\usepackage{xcolor}
\usepackage{tikz}
\usepackage{pgfplots}
\definecolor{c01}{HTML}{DC050C} 
\definecolor{c02}{HTML}{5289C7} 
\definecolor{c03}{HTML}{4EB256} 
\definecolor{c04}{HTML}{882E72} 
\definecolor{c05}{HTML}{CC4C02}
\definecolor{c06}{HTML}{AA3377}
\definecolor{c07}{HTML}{BBBBBB}
\definecolor{c08}{HTML}{BBBBBB}

\headers{}{}

\title{Automated Discovery of Branching Rules with Optimal Complexity for the Maximum Independent Set Problem
  \thanks{Submitted to the editors DATE.
  \funding{
    This work is partially funded by the National Key R\&D Program of China (Grant No. 2024YFE0102500), National Natural Science Foundation of China (No. 12404568, 12047503, 12325501 and 12247104), the Guangzhou Municipal Science and Technology Project (No. 2023A03J00904), and project ZDRW-XX-2022-3-02 of the Chinese Academy of Sciences. P.Z. is partially supported by the Innovation Program for Quantum Science and Technology project 2021ZD0301900.
  }
  }
}

\author{
    Xuan-Zhao Gao \thanks{ Thrust of Advanced Materials, The Hong Kong University of Science and Technology (Guangzhou), Guangdong, China; and Department of Physics, The Hong Kong University of Science and Technology, Hong Kong SAR, China (\email{xz.gao@connect.ust.hk}).}
    \and Yi-Jia Wang \thanks{CAS Key Laboratory for Theoretical Physics, Institute of Theoretical Physics, Chinese Academy of Sciences, Beijing 100190, China; and School of Physical Sciences, University of Chinese Academy of Sciences, Beijing 100049, China (\email{wangyijia@itp.ac.cn}), contributed equally with Xuan-Zhao Gao to this work.}
    \and Pan Zhang \thanks{CAS Key Laboratory for Theoretical Physics, Institute of Theoretical Physics, Chinese Academy of Sciences, Beijing 100190, China; and School of Fundamental Physics and Mathematical Sciences, Hangzhou Institute for Advanced Study, UCAS, Hangzhou 310024, China (\email{panzhang@itp.ac.cn}).}
    \and Jin-Guo Liu \thanks{ Thrust of Advanced Materials, The Hong Kong University of Science and Technology (Guangzhou), Guangdong, China (\email{jinguoliu@hkust-gz.edu.cn}).}
}
\usepackage{amsopn}


%% file: images/3rr_mean.tex
\begin{tikzpicture}
\begin{axis}[height={160pt}, width={180pt}, title={3-regular graphs}, xlabel={Number of vertices}, xmajorgrids={true}, ymode={log}, mark size={3.5pt}, ymajorgrids={true}, yminorticks={false}, ylabel={Number of branches}, label style={font={\footnotesize}}, tick label style={font={\footnotesize}}, scatter/classes={xiao2013={mark={asterisk}}, obfold={mark={o}}, obxiao={mark={triangle}}}, legend style={legend columns={4}, at={(0.51,-0.26)
}, anchor={south}, draw={none}, font={\footnotesize}, column sep={1.5}}]
    \addplot[c03, scatter, only marks, scatter src={explicit symbolic}, legend image post style={black}, legend style={text={black}, font={\footnotesize}}]
        table[row sep={\\}, meta={label}]
        {
            x  y  label  \\
            60  5.155602358173561  xiao2013
              \\
            80  12.853840950699306  xiao2013
              \\
            100  32.856746608044425  xiao2013
              \\
            120  84.06976456288305  xiao2013
              \\
            140  216.61709644876703  xiao2013
              \\
            160  569.8487885376636  xiao2013
              \\
            180  1477.348008964361  xiao2013
              \\
            200  3827.940711984912  xiao2013
              \\
            220  9754.763596282904  xiao2013
              \\
        }
        ;
    \addplot[c01, scatter, only marks, scatter src={explicit symbolic}]
        table[row sep={\\}, meta={label}]
        {
            x  y  label  \\
            60  8.547331411171319  obfold
              \\
            80  21.462494969378344  obfold
              \\
            100  53.22108436313227  obfold
              \\
            120  128.98817314035972  obfold
              \\
            140  314.2623742655755  obfold
              \\
            160  776.257895894933  obfold
              \\
            180  1919.5877577580188  obfold
              \\
            200  4603.315779042051  obfold
              \\
            220  11198.308958511976  obfold
              \\
        }
        ;
    \addplot[c01, scatter, only marks, scatter src={explicit symbolic}]
        table[row sep={\\}, meta={label}]
        {
            x  y  label  \\
            60  4.6478746452399875  obxiao
              \\
            80  11.814279842997914  obxiao
              \\
            100  27.620878387908785  obxiao
              \\
            120  64.7205689335533  obxiao
              \\
            140  151.48837880883337  obxiao
              \\
            160  362.12520363313956  obxiao
              \\
            180  879.1384360748485  obxiao
              \\
            200  2055.1066794923768  obxiao
              \\
            220  4910.64458551355  obxiao
              \\
        }
        ;
\end{axis}
\end{tikzpicture}

%% file: images/er_mean.tex
\begin{tikzpicture}
\begin{axis}[height={160pt}, width={180pt}, title={Erdos-Renyi graphs}, xlabel={Number of vertices}, xmajorgrids={true}, mark size={3.5pt}, ymode={log}, ymajorgrids={true}, yminorticks={false}, ylabel={Number of branches}, label style={font={\footnotesize}}, tick label style={font={\footnotesize}}, scatter/classes={ipfold={mark={o}}, ipxiao={mark={triangle}}, vcfold={mark={x}}, vcpack={mark={square}}}, legend style={legend columns={4}, at={(0.51,-0.26)
}, anchor={south}, draw={none}, font={\footnotesize}, column sep={1.5}}]
    \addplot[c01, scatter, only marks, scatter src={explicit symbolic}]
        table[row sep={\\}, meta={label}]
        {
            x  y  label  \\
            100  1.1914848973116265  ipfold
              \\
            200  1.2765465569253474  ipfold
              \\
            300  1.429798706159736  ipfold
              \\
            400  1.588216808829317  ipfold
              \\
            500  1.8059240774740366  ipfold
              \\
            600  2.0272345953317275  ipfold
              \\
            700  2.313183014978612  ipfold
              \\
            800  2.4225675481124784  ipfold
              \\
            900  2.732299866744741  ipfold
              \\
            1000  3.160620788857027  ipfold
              \\
        }
        ;
    \addplot[c01, scatter, only marks, scatter src={explicit symbolic}]
        table[row sep={\\}, meta={label}]
        {
            x  y  label  \\
            100  1.0034717485095028  ipxiao
              \\
            200  1.0108085516907024  ipxiao
              \\
            300  1.0266189351150794  ipxiao
              \\
            400  1.0336039239297363  ipxiao
              \\
            500  1.051406614310668  ipxiao
              \\
            600  1.0686949278656799  ipxiao
              \\
            700  1.0742642636856516  ipxiao
              \\
            800  1.0976311481572165  ipxiao
              \\
            900  1.1388173529920471  ipxiao
              \\
            1000  1.1520897039266373  ipxiao
              \\
        }
        ;
    \addplot[c02, scatter, only marks, scatter src={explicit symbolic}, legend image post style={black}, legend style={text={black}, font={\footnotesize}}]
        table[row sep={\\}, meta={label}]
        {
            x  y  label  \\
            100  3.324601352738916  vcfold
              \\
            200  6.016730202490998  vcfold
              \\
            300  10.937512757622054  vcfold
              \\
            400  17.312001168353763  vcfold
              \\
            500  26.722037094504643  vcfold
              \\
            600  47.306984336487474  vcfold
              \\
            700  70.66032937960429  vcfold
              \\
            800  104.74493319628976  vcfold
              \\
            900  154.42374334491666  vcfold
              \\
            1000  236.30701465529503  vcfold
              \\
        }
        ;
    \addplot[c02, scatter, only marks, scatter src={explicit symbolic}]
        table[row sep={\\}, meta={label}]
        {
            x  y  label  \\
            100  1.0038787037880033  vcpack
              \\
            200  1.0113124673763194  vcpack
              \\
            300  1.0199419580511582  vcpack
              \\
            400  1.024823898028535  vcpack
              \\
            500  1.042489608887938  vcpack
              \\
            600  1.0539663909581074  vcpack
              \\
            700  1.05781834271155  vcpack
              \\
            800  1.0785732163103672  vcpack
              \\
            900  1.096265575504883  vcpack
              \\
            1000  1.11175086472622  vcpack
              \\
        }
        ;
\end{axis}
\end{tikzpicture}

%% file: images/ksg_mean.tex
\begin{tikzpicture}
\begin{axis}[height={160pt}, width={180pt}, title={King's subgraphs}, xlabel={Number of vertices}, xmajorgrids={true}, mark size={3.5pt}, ymode={log}, ymajorgrids={true}, yminorticks={false}, ylabel={Number of branches}, label style={font={\footnotesize}}, tick label style={font={\footnotesize}}, scatter/classes={ipfold={mark={o}}, ipxiao={mark={triangle}}, vcfold={mark={x}}, vcpack={mark={square}}}, legend style={legend columns={4}, at={(0.51,-0.26)
}, anchor={south}, draw={none}, font={\footnotesize}, column sep={1.5}}]
    \addplot[c01, scatter, only marks, scatter src={explicit symbolic}]
        table[row sep={\\}, meta={label}]
        {
            x  y  label  \\
            80  1.0038787037880033  ipfold
              \\
            180  1.1391114722921258  ipfold
              \\
            320  2.104611846648551  ipfold
              \\
            500  10.24809660393024  ipfold
              \\
            720  206.73703276213993  ipfold
              \\
        }
        ;
    \addplot[c01, scatter, only marks, scatter src={explicit symbolic}]
        table[row sep={\\}, meta={label}]
        {
            x  y  label  \\
            80  1.0  ipxiao
              \\
            180  1.0  ipxiao
              \\
            320  1.024503428619609  ipxiao
              \\
            500  1.2330241635884165  ipxiao
              \\
            720  2.4502438651732428  ipxiao
              \\
        }
        ;
    \addplot[c02, scatter, only marks, scatter src={explicit symbolic}, legend image post style={black}, legend style={text={black}, font={\footnotesize}}]
        table[row sep={\\}, meta={label}]
        {
            x  y  label  \\
            80  8.98265754177741  vcfold
              \\
            180  167.55918454651427  vcfold
              \\
            320  10614.399881794405  vcfold
              \\
            500  3.2492125404101643e6  vcfold
              \\
        }
        ;
    \addplot[c02, scatter, only marks, scatter src={explicit symbolic}]
        table[row sep={\\}, meta={label}]
        {
            x  y  label  \\
            80  1.0  vcpack
              \\
            180  1.0  vcpack
              \\
            320  1.0198218332656155  vcpack
              \\
            500  1.2047680919704655  vcpack
              \\
            720  2.1202532982512414  vcpack
              \\
        }
        ;
\end{axis}
\end{tikzpicture}

%% file: images/sq_mean.tex
\begin{tikzpicture}
\begin{axis}[height={160pt}, width={180pt}, title={Grid graphs}, xlabel={Number of vertices}, xmajorgrids={true}, mark size={3.5pt}, ymode={log}, ymajorgrids={true}, yminorticks={false}, ylabel={Number of branches}, label style={font={\footnotesize}}, tick label style={font={\footnotesize}}, scatter/classes={ipfold={mark={o}}, ipxiao={mark={triangle}}, vcfold={mark={x}}, vcpack={mark={square}}}, legend style={legend columns={4}, at={(0.51,-0.26)
}, anchor={south}, draw={none}, font={\footnotesize}, column sep={1.5}}]
    \addplot[c01, scatter, only marks, scatter src={explicit symbolic}]
        table[row sep={\\}, meta={label}]
        {
            x  y  label  \\
            80  1.0418661359379204  ipfold
              \\
            180  1.11492447912701  ipfold
              \\
            320  1.2992216822427385  ipfold
              \\
            500  1.548349694599786  ipfold
              \\
            720  1.985646481463284  ipfold
              \\
            980  2.8474783787138014  ipfold
              \\
            1280  4.017515216696855  ipfold
              \\
            1620  5.7793401662955395  ipfold
              \\
            2000  10.281781980148734  ipfold
              \\
        }
        ;
    \addplot[c01, scatter, only marks, scatter src={explicit symbolic}]
        table[row sep={\\}, meta={label}]
        {
            x  y  label  \\
            80  1.0252672378885939  ipxiao
              \\
            180  1.0765643727738674  ipxiao
              \\
            320  1.2128795642448666  ipxiao
              \\
            500  1.3883154092510803  ipxiao
              \\
            720  1.6676945971332373  ipxiao
              \\
            980  2.1348044166135725  ipxiao
              \\
            1280  2.941720898597526  ipxiao
              \\
            1620  3.8931587726155197  ipxiao
              \\
            2000  5.864607995647113  ipxiao
              \\
        }
        ;
    \addplot[c02, scatter, only marks, scatter src={explicit symbolic}, legend image post style={black}, legend style={text={black}, font={\footnotesize}}]
        table[row sep={\\}, meta={label}]
        {
            x  y  label  \\
            80  8.98265754177741  vcfold
              \\
            180  167.55918454651427  vcfold
              \\
            320  10614.399881794405  vcfold
              \\
        }
        ;
    \addplot[c02, scatter, only marks, scatter src={explicit symbolic}]
        table[row sep={\\}, meta={label}]
        {
            x  y  label  \\
            80  1.0328172219426814  vcpack
              \\
            180  1.0805564459888153  vcpack
              \\
            320  1.2117015913362839  vcpack
              \\
            500  1.3810060423633534  vcpack
              \\
            720  1.6187728608334593  vcpack
              \\
            980  2.0131747785523286  vcpack
              \\
            1280  2.5260324063303563  vcpack
              \\
            1620  3.058645780512276  vcpack
              \\
            2000  3.9794915019899957  vcpack
              \\
        }
        ;
\end{axis}
\end{tikzpicture}

%% file: legend_compare.tex
\begin{tikzpicture}[every mark/.append style={mark size=3.5pt}]
    \begin{axis}[
      hide axis,
      height=46px,
      xmin=0,
      xmax=1,
      ymin=0,
      ymax=1,
      legend style={
        legend columns={6},
        at={(0.51,1.0)},
        anchor={south},
        draw={none},
        font={\footnotesize},
        column sep={5.0}
      }
      ]
      \addlegendimage{mark={asterisk}, color=c03, only marks}
      \addlegendentry{\texttt{xiao2013}}
      \addlegendimage{mark={o}, color=c01, only marks}
      \addlegendentry{\texttt{ob}}
      \addlegendimage{mark={triangle}, color=c01, only marks}
      \addlegendentry{\texttt{ob+xiao}}
      \addlegendimage{mark={x}, color=c02, only marks}
      \addlegendentry{\texttt{akiba2015}}
      \addlegendimage{mark={square}, color=c02, only marks}
      \addlegendentry{\texttt{akiba2015+xiao\&packing}}
  
    \end{axis}
  \end{tikzpicture}
  

%% file: images/3rr_mean_ip_lp.tex
\begin{tikzpicture}
\begin{axis}[height={160pt}, width={180pt}, title={3-regular graphs}, xlabel={Number of vertices}, xmajorgrids={true}, ymode={log}, mark size={3.5pt}, ymajorgrids={true}, yminorticks={false}, ylabel={Number of branches}, label style={font={\footnotesize}}, tick label style={font={\footnotesize}}, scatter/classes={obfold={mark={o}}, obxiao={mark={triangle}}, obrelaxfold={mark={diamond}}, obrelaxxiao={mark={pentagon}}}, legend style={legend columns={4}, at={(0.51,-0.26)
}, anchor={south}, draw={none}, font={\footnotesize}, column sep={1.5}}]
    \addplot[c01, scatter, only marks, scatter src={explicit symbolic}]
        table[row sep={\\}, meta={label}]
        {
            x  y  label  \\
            60  8.547331411171319  obfold
              \\
            80  21.462494969378344  obfold
              \\
            100  53.22108436313227  obfold
              \\
            120  128.98817314035972  obfold
              \\
            140  314.2623742655755  obfold
              \\
            160  776.257895894933  obfold
              \\
            180  1919.5877577580188  obfold
              \\
            200  4603.315779042051  obfold
              \\
            220  11198.308958511976  obfold
              \\
        }
        ;
    \addplot[c01, scatter, only marks, scatter src={explicit symbolic}]
        table[row sep={\\}, meta={label}]
        {
            x  y  label  \\
            60  4.6478746452399875  obxiao
              \\
            80  11.814279842997914  obxiao
              \\
            100  27.620878387908785  obxiao
              \\
            120  64.7205689335533  obxiao
              \\
            140  151.48837880883337  obxiao
              \\
            160  362.12520363313956  obxiao
              \\
            180  879.1384360748485  obxiao
              \\
            200  2055.1066794923768  obxiao
              \\
            220  4910.64458551355  obxiao
              \\
        }
        ;
    \addplot[c02, scatter, only marks, scatter src={explicit symbolic}]
        table[row sep={\\}, meta={label}]
        {
            x  y  label  \\
            60  10.197907237228197  obrelaxfold
              \\
            80  25.644970080036096  obrelaxfold
              \\
            100  65.48657772035509  obrelaxfold
              \\
            120  157.39239439237215  obrelaxfold
              \\
            140  386.06805912914433  obrelaxfold
              \\
            160  963.7134513014173  obrelaxfold
              \\
            180  2413.6568852122386  obrelaxfold
              \\
            200  5823.620376026119  obrelaxfold
              \\
            220  14275.813967905267  obrelaxfold
              \\
        }
        ;
    \addplot[c02, scatter, only marks, scatter src={explicit symbolic}]
        table[row sep={\\}, meta={label}]
        {
            x  y  label  \\
            60  5.2146332511851154  obrelaxxiao
              \\
            80  13.483713573907071  obrelaxxiao
              \\
            100  32.14029554556289  obrelaxxiao
              \\
            120  76.38421599409541  obrelaxxiao
              \\
            140  181.9005755316565  obrelaxxiao
              \\
            160  436.1486648623544  obrelaxxiao
              \\
            180  1068.5657748398262  obrelaxxiao
              \\
            200  2530.6732212174907  obrelaxxiao
              \\
            220  6103.662608264964  obrelaxxiao
              \\
        }
        ;
\end{axis}
\end{tikzpicture}

%% file: images/er_mean_ip_lp.tex
\begin{tikzpicture}
\begin{axis}[height={160pt}, width={180pt}, title={Erdos-Renyi graphs}, xlabel={Number of vertices}, xmajorgrids={true}, ymax={15}, mark size={3.5pt}, ymode={log}, ymajorgrids={true}, yminorticks={false}, ylabel={Number of branches}, label style={font={\footnotesize}}, tick label style={font={\footnotesize}}, scatter/classes={ipfold={mark={o}}, ipxiao={mark={triangle}}, iprelaxfold={mark={diamond}}, iprelaxxiao={mark={pentagon}}}, legend style={legend columns={4}, at={(0.51,-0.26)
}, anchor={south}, draw={none}, font={\footnotesize}, column sep={1.5}}]
    \addplot[c01, scatter, only marks, scatter src={explicit symbolic}]
        table[row sep={\\}, meta={label}]
        {
            x  y  label  \\
            100  1.1914848973116265  ipfold
              \\
            200  1.2765465569253474  ipfold
              \\
            300  1.429798706159736  ipfold
              \\
            400  1.588216808829317  ipfold
              \\
            500  1.8059240774740366  ipfold
              \\
            600  2.0272345953317275  ipfold
              \\
            700  2.313183014978612  ipfold
              \\
            800  2.4225675481124784  ipfold
              \\
            900  2.732299866744741  ipfold
              \\
            1000  3.160620788857027  ipfold
              \\
        }
        ;
    \addplot[c01, scatter, only marks, scatter src={explicit symbolic}]
        table[row sep={\\}, meta={label}]
        {
            x  y  label  \\
            100  1.0034717485095028  ipxiao
              \\
            200  1.0108085516907024  ipxiao
              \\
            300  1.0266189351150794  ipxiao
              \\
            400  1.0336039239297363  ipxiao
              \\
            500  1.051406614310668  ipxiao
              \\
            600  1.0686949278656799  ipxiao
              \\
            700  1.0742642636856516  ipxiao
              \\
            800  1.0976311481572165  ipxiao
              \\
            900  1.1388173529920471  ipxiao
              \\
            1000  1.1520897039266373  ipxiao
              \\
        }
        ;
    \addplot[c02, scatter, only marks, scatter src={explicit symbolic}]
        table[row sep={\\}, meta={label}]
        {
            x  y  label  \\
            100  1.1911421777665712  iprelaxfold
              \\
            200  1.2765465569253474  iprelaxfold
              \\
            300  1.429897355481133  iprelaxfold
              \\
            400  1.588216808829317  iprelaxfold
              \\
            500  1.8059240774740366  iprelaxfold
              \\
            600  2.0272345953317275  iprelaxfold
              \\
            700  2.313183014978612  iprelaxfold
              \\
            800  2.422612831490539  iprelaxfold
              \\
            900  2.732299866744741  iprelaxfold
              \\
            1000  3.160620788857027  iprelaxfold
              \\
        }
        ;
    \addplot[c02, scatter, only marks, scatter src={explicit symbolic}]
        table[row sep={\\}, meta={label}]
        {
            x  y  label  \\
            100  1.0034717485095028  iprelaxxiao
              \\
            200  1.0108085516907024  iprelaxxiao
              \\
            300  1.027330779813536  iprelaxxiao
              \\
            400  1.0336039239297363  iprelaxxiao
              \\
            500  1.051406614310668  iprelaxxiao
              \\
            600  1.0686949278656799  iprelaxxiao
              \\
            700  1.0742642636856516  iprelaxxiao
              \\
            800  1.0976311481572165  iprelaxxiao
              \\
            900  1.1388173529920471  iprelaxxiao
              \\
            1000  1.1520897039266373  iprelaxxiao
              \\
        }
        ;
\end{axis}
\end{tikzpicture}

%% file: images/ksg_mean_ip_lp.tex
\begin{tikzpicture}
\begin{axis}[height={160pt}, width={180pt}, title={King's subgraphs}, xlabel={Number of vertices}, xmajorgrids={true}, mark size={3.5pt}, ymode={log}, ymajorgrids={true}, yminorticks={false}, ylabel={Number of branches}, label style={font={\footnotesize}}, tick label style={font={\footnotesize}}, scatter/classes={ipfold={mark={o}}, ipxiao={mark={triangle}}, lpfold={mark={diamond}}, lpxiao={mark={pentagon}}}, legend style={legend columns={4}, at={(0.51,-0.26)
}, anchor={south}, draw={none}, font={\footnotesize}, column sep={1.5}}]
    \addplot[c01, scatter, only marks, scatter src={explicit symbolic}]
        table[row sep={\\}, meta={label}]
        {
            x  y  label  \\
            80  1.0038787037880033  ipfold
              \\
            180  1.1391114722921258  ipfold
              \\
            320  2.104611846648551  ipfold
              \\
            500  10.24809660393024  ipfold
              \\
            720  206.73703276213993  ipfold
              \\
        }
        ;
    \addplot[c01, scatter, only marks, scatter src={explicit symbolic}]
        table[row sep={\\}, meta={label}]
        {
            x  y  label  \\
            80  1.0  ipxiao
              \\
            180  1.0  ipxiao
              \\
            320  1.024503428619609  ipxiao
              \\
            500  1.2330241635884165  ipxiao
              \\
            720  2.4502438651732428  ipxiao
              \\
        }
        ;
    \addplot[c02, scatter, only marks, scatter src={explicit symbolic}]
        table[row sep={\\}, meta={label}]
        {
            x  y  label  \\
            80  1.0038787037880033  lpfold
              \\
            180  1.1506086143164505  lpfold
              \\
            320  2.286360506070147  lpfold
              \\
            500  13.70829084592613  lpfold
              \\
            720  395.32506166592134  lpfold
              \\
        }
        ;
    \addplot[c02, scatter, only marks, scatter src={explicit symbolic}]
        table[row sep={\\}, meta={label}]
        {
            x  y  label  \\
            80  1.0  lpxiao
              \\
            180  1.0  lpxiao
              \\
            320  1.02544260182817  lpxiao
              \\
            500  1.2686416725584773  lpxiao
              \\
            720  2.7856829698399665  lpxiao
              \\
        }
        ;
\end{axis}
\end{tikzpicture}

%% file: images/sq_mean_ip_lp.tex
\begin{tikzpicture}
\begin{axis}[height={160pt}, width={180pt}, title={Grid graphs}, xlabel={Number of vertices}, xmajorgrids={true}, mark size={3.5pt}, ymode={log}, ymajorgrids={true}, yminorticks={false}, ylabel={Number of branches}, label style={font={\footnotesize}}, tick label style={font={\footnotesize}}, scatter/classes={ipfold={mark={o}}, ipxiao={mark={triangle}}, lpfold={mark={diamond}}, lpxiao={mark={pentagon}}}, legend style={legend columns={4}, at={(0.51,-0.26)
}, anchor={south}, draw={none}, font={\footnotesize}, column sep={1.5}}]
    \addplot[c01, scatter, only marks, scatter src={explicit symbolic}]
        table[row sep={\\}, meta={label}]
        {
            x  y  label  \\
            80  1.0418661359379204  ipfold
              \\
            180  1.11492447912701  ipfold
              \\
            320  1.2992216822427385  ipfold
              \\
            500  1.548349694599786  ipfold
              \\
            720  1.985646481463284  ipfold
              \\
            980  2.8474783787138014  ipfold
              \\
            1280  4.017515216696855  ipfold
              \\
            1620  5.7793401662955395  ipfold
              \\
            2000  10.281781980148734  ipfold
              \\
        }
        ;
    \addplot[c01, scatter, only marks, scatter src={explicit symbolic}]
        table[row sep={\\}, meta={label}]
        {
            x  y  label  \\
            80  1.0252672378885939  ipxiao
              \\
            180  1.0765643727738674  ipxiao
              \\
            320  1.2128795642448666  ipxiao
              \\
            500  1.3883154092510803  ipxiao
              \\
            720  1.6676945971332373  ipxiao
              \\
            980  2.1348044166135725  ipxiao
              \\
            1280  2.941720898597526  ipxiao
              \\
            1620  3.8931587726155197  ipxiao
              \\
            2000  5.864607995647113  ipxiao
              \\
        }
        ;
    \addplot[c02, scatter, only marks, scatter src={explicit symbolic}]
        table[row sep={\\}, meta={label}]
        {
            x  y  label  \\
            80  1.0418661359379204  lpfold
              \\
            180  1.11492447912701  lpfold
              \\
            320  1.2995954987964426  lpfold
              \\
            500  1.550904500101598  lpfold
              \\
            720  1.9916805329223781  lpfold
              \\
            980  2.857578956086256  lpfold
              \\
            1280  4.042214869438546  lpfold
              \\
            1620  5.809110878852561  lpfold
              \\
            2000  10.412139925909354  lpfold
              \\
        }
        ;
    \addplot[c02, scatter, only marks, scatter src={explicit symbolic}]
        table[row sep={\\}, meta={label}]
        {
            x  y  label  \\
            80  1.0252672378885939  lpxiao
              \\
            180  1.0779882688348652  lpxiao
              \\
            320  1.2148331902969913  lpxiao
              \\
            500  1.3931352912038675  lpxiao
              \\
            720  1.674608066502247  lpxiao
              \\
            980  2.1465836173258275  lpxiao
              \\
            1280  2.9685117218113333  lpxiao
              \\
            1620  3.940878005875321  lpxiao
              \\
            2000  5.967869136933446  lpxiao
              \\
        }
        ;
\end{axis}
\end{tikzpicture}

%% file: legend_compare_iplp.tex
\begin{tikzpicture}[every mark/.append style={mark size=3.5pt}]
    \begin{axis}[
      hide axis,
      height=46px,
      xmin=0,
      xmax=1,
      ymin=0,
      ymax=1,
      legend style={
        legend columns={6},
        at={(0.51,1.0)},
        anchor={south},
        draw={none},
        font={\footnotesize},
        column sep={5.0}
      }
      ]
      \addlegendimage{mark={o}, color=c01, only marks}
      \addlegendentry{\texttt{ob}}
      \addlegendimage{mark={triangle}, color=c01, only marks}
      \addlegendentry{\texttt{ob+xiao}}
      \addlegendimage{mark={diamond}, color=c02, only marks}
      \addlegendentry{\texttt{ob\_relax}}
      \addlegendimage{mark={pentagon}, color=c02, only marks}
      \addlegendentry{\texttt{ob\_relax+xiao}}
  
    \end{axis}
  \end{tikzpicture}
  

%% file: images/3rr_worst.tex
\begin{tikzpicture}
\begin{axis}[height={160pt}, width={180pt}, title={3-regular graphs}, xlabel={Number of vertices}, xmajorgrids={true}, ymode={log}, mark size={3.5pt}, ymajorgrids={true}, yminorticks={false}, ylabel={Number of branches}, label style={font={\footnotesize}}, tick label style={font={\footnotesize}}, scatter/classes={xiao2013={mark={asterisk}}, obfold={mark={o}}, obxiao={mark={triangle}}}, legend style={legend columns={4}, at={(0.51,-0.26)
}, anchor={south}, draw={none}, font={\footnotesize}, column sep={1.5}}]
    \addplot[c03, scatter, only marks, scatter src={explicit symbolic}, legend image post style={black}, legend style={text={black}, font={\footnotesize}}]
        table[row sep={\\}, meta={label}]
        {
            x  y  label  \\
            60  12  xiao2013
              \\
            80  30  xiao2013
              \\
            100  70  xiao2013
              \\
            120  187  xiao2013
              \\
            140  446  xiao2013
              \\
            160  1029  xiao2013
              \\
            180  2960  xiao2013
              \\
            200  7457  xiao2013
              \\
            220  19055  xiao2013
              \\
        }
        ;
    \addplot[c01, scatter, only marks, scatter src={explicit symbolic}]
        table[row sep={\\}, meta={label}]
        {
            x  y  label  \\
            60  16  obfold
              \\
            80  41  obfold
              \\
            100  103  obfold
              \\
            120  237  obfold
              \\
            140  622  obfold
              \\
            160  1665  obfold
              \\
            180  3823  obfold
              \\
            200  8665  obfold
              \\
            220  22018  obfold
              \\
        }
        ;
    \addplot[c01, scatter, only marks, scatter src={explicit symbolic}]
        table[row sep={\\}, meta={label}]
        {
            x  y  label  \\
            60  12  obxiao
              \\
            80  28  obxiao
              \\
            100  58  obxiao
              \\
            120  138  obxiao
              \\
            140  299  obxiao
              \\
            160  779  obxiao
              \\
            180  1755  obxiao
              \\
            200  3971  obxiao
              \\
            220  9424  obxiao
              \\
        }
        ;
\end{axis}
\end{tikzpicture}

%% file: images/er_worst.tex
\begin{tikzpicture}
\begin{axis}[height={160pt}, width={180pt}, title={Erdos-Renyi graphs}, xlabel={Number of vertices}, xmajorgrids={true}, mark size={3.5pt}, ymode={log}, ymajorgrids={true}, yminorticks={false}, ylabel={Number of branches}, label style={font={\footnotesize}}, tick label style={font={\footnotesize}}, scatter/classes={ipfold={mark={o}}, ipxiao={mark={triangle}}, vcfold={mark={x}}, vcpack={mark={square}}}, legend style={legend columns={4}, at={(0.51,-0.26)
}, anchor={south}, draw={none}, font={\footnotesize}, column sep={1.5}}]
    \addplot[c01, scatter, only marks, scatter src={explicit symbolic}]
        table[row sep={\\}, meta={label}]
        {
            x  y  label  \\
            100  11  ipfold
              \\
            200  25  ipfold
              \\
            300  58  ipfold
              \\
            400  191  ipfold
              \\
            500  207  ipfold
              \\
            600  602  ipfold
              \\
            700  259  ipfold
              \\
            800  1423  ipfold
              \\
            900  10554  ipfold
              \\
            1000  1378  ipfold
              \\
        }
        ;
    \addplot[c01, scatter, only marks, scatter src={explicit symbolic}]
        table[row sep={\\}, meta={label}]
        {
            x  y  label  \\
            100  4  ipxiao
              \\
            200  6  ipxiao
              \\
            300  11  ipxiao
              \\
            400  28  ipxiao
              \\
            500  35  ipxiao
              \\
            600  76  ipxiao
              \\
            700  17  ipxiao
              \\
            800  64  ipxiao
              \\
            900  387  ipxiao
              \\
            1000  107  ipxiao
              \\
        }
        ;
    \addplot[c02, scatter, only marks, scatter src={explicit symbolic}, legend image post style={black}, legend style={text={black}, font={\footnotesize}}]
        table[row sep={\\}, meta={label}]
        {
            x  y  label  \\
            100  51  vcfold
              \\
            200  633  vcfold
              \\
            300  2851  vcfold
              \\
            400  16960  vcfold
              \\
            500  25399  vcfold
              \\
            600  165504  vcfold
              \\
            700  137216  vcfold
              \\
            800  1841898  vcfold
              \\
            900  5561262  vcfold
              \\
            1000  2430518  vcfold
              \\
        }
        ;
    \addplot[c02, scatter, only marks, scatter src={explicit symbolic}]
        table[row sep={\\}, meta={label}]
        {
            x  y  label  \\
            100  4  vcpack
              \\
            200  5  vcpack
              \\
            300  7  vcpack
              \\
            400  15  vcpack
              \\
            500  12  vcpack
              \\
            600  14  vcpack
              \\
            700  12  vcpack
              \\
            800  16  vcpack
              \\
            900  35  vcpack
              \\
            1000  21  vcpack
              \\
        }
        ;
\end{axis}
\end{tikzpicture}

%% file: images/ksg_worst.tex
\begin{tikzpicture}
\begin{axis}[height={160pt}, width={180pt}, title={King's subgraphs}, xlabel={Number of vertices}, xmajorgrids={true}, mark size={3.5pt}, ymode={log}, ymajorgrids={true}, yminorticks={false}, ylabel={Number of branches}, label style={font={\footnotesize}}, tick label style={font={\footnotesize}}, scatter/classes={ipfold={mark={o}}, ipxiao={mark={triangle}}, vcfold={mark={x}}, vcpack={mark={square}}}, legend style={legend columns={4}, at={(0.51,-0.26)
}, anchor={south}, draw={none}, font={\footnotesize}, column sep={1.5}}]
    \addplot[c01, scatter, only marks, scatter src={explicit symbolic}]
        table[row sep={\\}, meta={label}]
        {
            x  y  label  \\
            80  3  ipfold
              \\
            180  16  ipfold
              \\
            320  119  ipfold
              \\
            500  2451  ipfold
              \\
            720  93547  ipfold
              \\
        }
        ;
    \addplot[c01, scatter, only marks, scatter src={explicit symbolic}]
        table[row sep={\\}, meta={label}]
        {
            x  y  label  \\
            80  1  ipxiao
              \\
            180  1  ipxiao
              \\
            320  8  ipxiao
              \\
            500  40  ipxiao
              \\
            720  337  ipxiao
              \\
        }
        ;
    \addplot[c02, scatter, only marks, scatter src={explicit symbolic}, legend image post style={black}, legend style={text={black}, font={\footnotesize}}]
        table[row sep={\\}, meta={label}]
        {
            x  y  label  \\
            80  86  vcfold
              \\
            180  5975  vcfold
              \\
            320  3293564  vcfold
              \\
            500  341702208  vcfold
              \\
        }
        ;
    \addplot[c02, scatter, only marks, scatter src={explicit symbolic}]
        table[row sep={\\}, meta={label}]
        {
            x  y  label  \\
            80  1  vcpack
              \\
            180  1  vcpack
              \\
            320  8  vcpack
              \\
            500  23  vcpack
              \\
            720  177  vcpack
              \\
        }
        ;
\end{axis}
\end{tikzpicture}

%% file: images/sq_worst.tex
\begin{tikzpicture}
\begin{axis}[height={160pt}, width={180pt}, title={Grid graphs}, xlabel={Number of vertices}, xmajorgrids={true}, mark size={3.5pt}, ymode={log}, ymajorgrids={true}, yminorticks={false}, ylabel={Number of branches}, label style={font={\footnotesize}}, tick label style={font={\footnotesize}}, scatter/classes={ipfold={mark={o}}, ipxiao={mark={triangle}}, vcfold={mark={x}}, vcpack={mark={square}}}, legend style={legend columns={4}, at={(0.51,-0.26)
}, anchor={south}, draw={none}, font={\footnotesize}, column sep={1.5}}]
    \addplot[c01, scatter, only marks, scatter src={explicit symbolic}]
        table[row sep={\\}, meta={label}]
        {
            x  y  label  \\
            80  3  ipfold
              \\
            180  5  ipfold
              \\
            320  8  ipfold
              \\
            500  17  ipfold
              \\
            720  39  ipfold
              \\
            980  92  ipfold
              \\
            1280  410  ipfold
              \\
            1620  874  ipfold
              \\
            2000  2306  ipfold
              \\
        }
        ;
    \addplot[c01, scatter, only marks, scatter src={explicit symbolic}]
        table[row sep={\\}, meta={label}]
        {
            x  y  label  \\
            80  2  ipxiao
              \\
            180  4  ipxiao
              \\
            320  7  ipxiao
              \\
            500  12  ipxiao
              \\
            720  24  ipxiao
              \\
            980  51  ipxiao
              \\
            1280  146  ipxiao
              \\
            1620  643  ipxiao
              \\
            2000  2269  ipxiao
              \\
        }
        ;
    \addplot[c02, scatter, only marks, scatter src={explicit symbolic}, legend image post style={black}, legend style={text={black}, font={\footnotesize}}]
        table[row sep={\\}, meta={label}]
        {
            x  y  label  \\
            80  86  vcfold
              \\
            180  5975  vcfold
              \\
            320  3293564  vcfold
              \\
        }
        ;
    \addplot[c02, scatter, only marks, scatter src={explicit symbolic}]
        table[row sep={\\}, meta={label}]
        {
            x  y  label  \\
            80  3  vcpack
              \\
            180  3  vcpack
              \\
            320  5  vcpack
              \\
            500  14  vcpack
              \\
            720  18  vcpack
              \\
            980  35  vcpack
              \\
            1280  36  vcpack
              \\
            1620  91  vcpack
              \\
            2000  125  vcpack
              \\
        }
        ;
\end{axis}
\end{tikzpicture}